\PassOptionsToPackage{dvipsnames}{xcolor}
\documentclass[final,onefignum,onetabnum]{siamart250211}

\usepackage{amsfonts}      %
\usepackage{mathrsfs}
\usepackage{bbm}
\usepackage{graphicx}
\usepackage{float}
\usepackage{xcolor}
\usepackage{enumerate}
\usepackage{tikz}
\usepackage[expansion=false]{microtype}
\usepackage{stmaryrd}
\usepackage{algorithm}
\usepackage{algpseudocode}
\usepackage[
  labelformat=parens,
  subrefformat=parens
]{subcaption}
\usepackage{comment}
\usepackage{mathtools}
\usepackage{autonum}

\hypersetup{
  pdftitle   = {A Benamou--Brenier proximal splitting method for constrained unbalanced optimal transport},
  pdfauthor  = {Mao Nishino, Martin Bauer, Tom Needham, Nicolas Charon},
  colorlinks = true,
  linkcolor  = MidnightBlue,
  citecolor  = Maroon,
  urlcolor   = Purple
}

\newsiamremark{remark}{Remark}

\newsiamremark{example}{Example}

\newcommand{\define}[1]{\textbf{#1}}
\newcommand{\R}{\mathbb{R}}

\newcommand{\eich}{H}
\newcommand{\Htx}{\eich(t,x)}

\newcommand{\ef}{F}
\newcommand{\Ft}{\ef(t)}
\newcommand{\Ce}{C}
\newcommand{\CEHF}{\mathcal{CE}_{\eich,\ef}(\rho_0,\rho_1)}
\newcommand{\AFF}{\mathcal{A}}

\newcommand{\AHFineq}{\AFF}
\newcommand{\Meas}{\mathcal{M}}
\newcommand{\Mplus}{\Meas^+}

\newcommand{\MplushC}{\Meas^+_{\eich,C}(\Omega)}

\newcommand{\WFRAF}{{\textrm{WFR}_{\delta}^{\eich,\ef}}}
\newcommand{\rhosq}{(\rho_0,\rho_1)^2}

\newcommand{\CE}{\mathcal{CE}(\rho_0,\rho_1)}

\newcommand{\AHF}{\AFF(\eich,\ef)}

\newcommand{\Gc}{\mathcal{G}_c}
\newcommand{\prox}{\textrm{prox}}

\headers{A Benamou--Brenier proximal splitting method for const. UOT}{Mao Nishino, Martin Bauer, Tom Needham, Nicolas Charon}

\title{A Benamou--Brenier proximal splitting method for constrained unbalanced optimal transport%
}

\author{Mao Nishino
\and
Martin Bauer
\and
Tom Needham
  \thanks{Department of Mathematics, Florida State University, Tallahassee, FL 32306, USA
  (\email{mnishino@fsu.edu}, \email{mbauer2@fsu.edu}, \email{tneedham@fsu.edu}).}
\and Nicolas Charon%
  \thanks{Department of Mathematics, University of Houston, Houston, TX 77204, USA
  (\email{ncharon@central.uh.edu}).}
}

\begin{document}

\maketitle
\begin{abstract}
The dynamic formulation of optimal transport, also known as the Benamou--Brenier formulation, has been extended to the unbalanced case by introducing a source term in the continuity equation. When this source term is penalized based on the Fisher-Rao metric, the resulting model is referred to as the Wasserstein-Fisher-Rao (WFR) setting, and allows for the comparison between any two positive measures without the need for equalized total mass. In  recent work we introduced a constrained variant of this model, in which affine integral equality constraints are imposed along the measure path. In the present paper, we propose a further generalization of this framework, which allows for constraints that apply not just to the density path but also to the momentum and source terms, and incorporates affine inequalities in addition to equality constraints. We prove, under suitable assumptions on the constraints, the well-posedness of the resulting class of convex variational problems. The paper is then primarily devoted to developing an effective numerical pipeline that tackles the corresponding constrained optimization problem based on finite difference discretizations and parallel proximal schemes. Our proposed  framework encompasses standard balanced and unbalanced optimal transport, as well as a multitude of natural and practically relevant constraints, and we highlight its versatility via several synthetic and real data examples.

\end{abstract}

\tableofcontents

\section{Introduction}
\typeout{TEXTWIDTH = \the\textwidth}
\typeout{TEXTHEIGHT = \the\textheight}
\typeout{PAPERWIDTH = \the\paperwidth}
\typeout{PAPERHEIGHT = \the\paperheight}
Optimal transport (OT) is a powerful mathematical framework for comparing probability distributions over a fixed metric space~\cite{villani2009optimal,peyre2019computational}. Classical formulations focus on the Wasserstein distances, a family of metrics that compare probability measures while incorporating the geometry of the underlying space. The theory of unbalanced optimal transport (UOT)  generalizes this framework by enabling the comparison of positive measures of potentially different total masses~\cite{chizat2018interpolating, ChizatEtAl2018JFA, PiccoliRossi2014,LieroMielkeSavare2016,LieroMielkeSavare2018Invent}. This flexibility is particularly useful in applications where differences in total population are important; for example, UOT has applications to the modeling of cell population dynamics~\cite{Schiebinger2019OptimalTransport}, image processing and color transfer~\cite{nguyen2023unbalanced}, crowd counting~\cite{Ma2021Learning} and domain adaptation in machine learning~\cite{fatras2021unbalanced}. Moreover, UOT can be useful even when total population differences are not explicitly important, as the ability to create or destroy mass leads to improved robustness to outlier noise~\cite{sejourne2019sinkhorn,wang2024outlierrobust}.

Intuitively, the dynamic formulation of UOT expands on the Benamou-Brenier characterization of classical OT~\cite{benamou_computational_2000}, thereby achieving the comparison of two measures by joining them via a path in the space of measures which minimizes a certain kinetic energy functional. From this perspective, it is natural in applications that one may wish to impose additional constraints on the mass transport path. For example, in population dynamics modeling, one may have knowledge of the total population over time, which can be formulated as a total mass constraint over the transport path. A particular instance of this constraint is that the total mass should remain constant while still allowing for local mass creation/destruction; this leads to the spherical Hellinger-Kantorovich (SHK) metric framework, introduced in~\cite{laschos2019geometric} and implemented via a deep learning approach in~\cite{jing2024machine}, with applications to generative sampling. Alternatively, one may wish to constrain the spatial characteristics of measure trajectories in the UOT framework. {Robotics applications~\cite{ding2025swarmdiffswarmrobotictrajectory}, for instance, involve obstacles that must be avoided along trajectories. Crowd motion models~\cite{butazzo2009, maury2010macroscopic, ruthotto2020machine} are often imposing congestion limits in specific regions in addition to obstacles. A more detailed discussion of several motivating application examples is provided in Section~\ref{sec:motivation} below.}

Inspired by such potential applications, a constrained unbalanced optimal transport model was considered in our previous work \cite{bauer2025path}, where we studied the theoretical aspects of the UOT problem with affine constraints on the mass density; we will refer to this framework as ACUOT, which stands for affine-constrained unbalanced optimal transport. To provide an informal overview of our formulation in \cite{bauer2025path}, let $\Omega \subset \R^n$ be a compact domain and let $\rho_0,\rho_1$ be two positive densities on $\Omega$. In dynamic UOT, one considers a path of densities {$\rho:[0,1] \times \Omega \to \R$}, a path of momentum vector fields $\omega:[0,1] \times \Omega \to \R^n$ (describing the transport dynamics), and a path of source terms $\zeta:[0,1] \times \Omega \to \R$ (capturing the creation or destruction of mass). Given a balance parameter $\delta \geq 0$, the UOT problem then seeks to minimize the Wasserstein-Fisher-Rao energy \cite{chizat2018interpolating}
\begin{equation}\label{eqn:informal_acuot}
    \int_0^1 \int_\Omega \frac{\|\omega(t,x)\|^2 + \delta^2 \zeta(t,x)^2}{2\rho(t,x)} \; dx \, dt,
\end{equation}
subject to the continuity equation
\(
    \partial_t \rho + \nabla \cdot \omega = \zeta,
\)
over all triples $(\rho,\omega,\zeta)$ satisfying the given initial data $\rho(0,\cdot)=\rho_0$, $\rho(1,\cdot) = \rho_1$. In the ACUOT problem, we additionally introduce an affine constraint on the density path
\begin{equation}\label{eqn:informal_constraints}
    \int_\Omega H(t,x) \; \rho(t,x) \; dx = F(t) \quad \mbox{for a.e.~$t \in [0,1]$},
\end{equation}
where $H:[0,1] \times \Omega \to \R^d$ and $F:[0,1] \to \R^d$ are given functions that define the affine constraint. In this formulation, the main contribution of \cite{bauer2025path} was to show the existence of minimizers for the ACUOT problem under suitable technical assumptions on the constraint functions $H$ and $F$. However, the above formulation lacks two important generalizations that are relevant in practice. First, the constraints are only imposed on the density path $\rho$, whereas it is natural in certain applications to impose constraints on the control variables as well; e.g., in modeling population dynamics, we may want to restrict the growth rate of the population (source term, $\zeta$) or limit the direction of mass movement (momentum, $\omega$). Second, the constraints are only affine integral equalities, while one may wish to instead impose inequality constraints; e.g., upper bounding the total population over time. Lastly, \cite{bauer2025path} only focused on the theoretical aspects of the model, such as the well-posedness of the corresponding variational problems, but did not address the numerical solution to the ACUOT problem in practice.

\subsection{Main contributions}
The goal of this paper is to address the limitations of the ACUOT model described above. Our main contributions are described as follows.

\noindent {\bf Generalized constrained UOT model.}
Including the control terms, we generalize the ACUOT model of \cite{bauer2025path} by considering the UOT problem described by \eqref{eqn:informal_acuot}, but with additional constraints of a form more general than \eqref{eqn:informal_constraints}. Namely, for $i = 1, \ldots, d$, we impose for a.e.~$t \in [0,1]$ that
\begin{align}\label{eqn:generalized_informal_acuot3}
    \int_\Omega H^{\rho}_i(t,x) \; \rho(t,x) \; dx &+ \int_\Omega H^{\omega}_i(t,x) \cdot \omega(t,x)dx + \int_\Omega H^{\zeta}_i(t,x) \; \zeta(t,x)dx \leq F_i(t).
\end{align}
 Here, $(H^{\rho}_i, H^{\omega}_i, H^{\zeta}_i):[0,1] \times \Omega \to \R\times \mathbb{R}^n \times \mathbb{R}$ are given functions that define the constraints on the density, momentum and source term, respectively, and $F_i:[0,1] \to \R$ is a given function that defines the bounds on the affine functional. A formal definition of this extended model is given in Section \ref{sec:generalized_model}. The main contribution{s} of this part of  the article {are} Theorem \ref{thm:main_theorem}, where we show the existence of minimizers for this generalized constrained UOT problem{, and Theorems \ref{thm:feasibility_density_constr} and \ref{thm:feasibility_controls_constr}, where we give general feasibility conditions under which our existence theorem holds}.

\noindent {\bf A discrete model.}
To develop a numerical method for solving the generalized constrained UOT problem, we first discretize its associated energy. This is achieved by employing a standard grid on the time--space domain $[0,1] \times \Omega$ and using finite-difference schemes to 
approximate the continuity equation. We then show in Theorem~\ref{thm:existence_discrete_wfr}, that the resulting discrete 
problem admits minimizers, under suitable assumptions, just as in the continuous setting.

\noindent {\bf Numerical framework for the generalized model.}
Finally, we use the above described discretization to develop a parallelizable numerical framework for solving the generalized constrained UOT problem. To make use of the convexity of the problem, we employ a convex optimization approach to solve the resulting discrete problem. In particular, we use the proximal splitting algorithm as done by \cite{chizat2018interpolating} but rely on the parallel proximal algorithm \cite{combettes-pesquet-ppxa-2008,combettes-pesquet-signal-2011} instead of the Douglas-Rachford algorithm used in the previous work. This algorithm allows us to incorporate multiple constraints easily and to parallelize the computation of the proximal operators. The previously obtained existence results for the discretized energy can be used to show the convergence of the algorithm, cf. Remark~\ref{rem:convergence}. Finally we  demonstrate the versatility of our numerical framework on several synthetic examples as well as real-world data on population dynamics.

\subsection{Relationship to existing literature} 
We first re-emphasize that the present paper deals with constraints for dynamic OT, in contrast with other lines of work which have instead introduced constraints on the transport plans in the static Kantorovich setting, such as capacity constraints~\cite{korman2015optimal,tang2024sinkhorn}, martingale constraints~\cite{ekren2018constrained,dolinsky2014martingale}, or constraints designed to handle Gaussian mixtures~\cite{delon2020wasserstein,wilson2023wasserstein}. Several past works have considered different constraints within dynamic formulations of optimal transport and proposed numerical approaches for such problems, although primarily in the balanced setting \cite{papadakis2014optimal,Wan2023ScalableDynamicOT,kerrache2022constrained,bredies2023generalized,WuRantzer_OMT_Nonlinear_InputDensity_2024}. In the unbalanced case, \cite{jing2024machine} develops a deep-learning-based algorithm for the specific SHK model. For more general constraints,  \cite{neklyudov2024wlf} and \cite{wan2024neural} both propose algorithms based on neural network architectures and with soft constraints enforced via penalty terms. All these works are geared towards applications to flow matching and sampling problems in high dimension, and, despite arguably being scalable for such purposes, they typically fall out of the convex optimization landscape, thus lacking clear theoretical convergence guarantees. Our work, on the other hand, focuses on a particular family of hard constraints---namely, affine integral equality/inequality constraints---which preserve the convexity of the original formulation of UOT. This allows us to provide conditions for the well-posedness of the resulting variational problems, and leverage convex optimization to derive a numerical solver with actual convergence properties. However, this makes the proposed approach better suited for optimal transport of density measures on low-dimensional domains. 

\subsection{{Motivating examples for constrained transport models}\label{sec:motivation}}
{Imposing constraints on the density, momentum, and source variables can be relevant in a variety of applications of optimal transport. We give several representative examples:\\
{\bf Single-cell microscopy.}
Time-lapse microscopy provides a sequence of images from which one can estimate
the spatial distribution of cells over time. Modern image-analysis pipelines
further extract cell trajectories and lineage information by solving the joint
segmentation and tracking problem \cite{Ulman2017,Maaska2021}. In this setting,
the density $\rho_t$ represents the spatial distribution of cells, the momentum
$\omega_t$ represents cell migration, and the source term $\zeta_t$ models local
cell proliferation, apoptosis, or cells entering and leaving the field of view.
While the initial and final density images determine the endpoint
distributions, they generally do not distinguish whether changes in local cell
density are caused by migration or by proliferation. However, tracked cell
trajectories provide direct information on the transport of cells, while lineage
reconstruction identifies cell division events. These observations naturally
lead to affine constraints on the transport and source variables. For example,
the number of cells crossing a prescribed interface during a given time interval
determines an integrated flux constraint on $\omega$, whereas the number of
mitotic events detected inside a region determines an integral constraint on
$\zeta$. Constrained dynamic unbalanced optimal transport therefore provides a
natural variational framework for combining density measurements with partial
trajectory and lineage information.\\
{\bf Inverse problems in imaging.} Several works, including \cite{Bredies_2020,bredies2023generalized}, have proposed to use balanced and unbalanced OT metrics as regularizer for inverse problems on measure spaces. One typical application is image reconstruction, in particular in dynamic MRI acquisition, where the image density at each time is only observed through a sparse set of samples in Fourier domain. As prescribing one Fourier coefficient amounts to a complex (or equivalently a two-dimensional real) affine integral constraint on the density, this does fit our considered framework. The model of \cite{Bredies_2020} however differs from the present work in two important ways: constraints are imposed via soft fidelity terms while the regularizer involves an additional total variation penalty. Our proposed approach could thus provide a potential alternative that does not require this total variation cost and in which one can incorporate further modeling constraints on the transport or source term.  \\
{\bf Transport-map methods for Bayesian inference.}
Transport-map methods seek deterministic transformations that push a simple
reference or prior distribution onto a posterior distribution, providing either
independent posterior samples or efficient proposals for Markov chain Monte
Carlo methods \cite{ElMoselhy2012}. For challenging inverse
problems, these transformations are often constructed progressively through
intermediate distributions, for example by tempering, multiscale
approximations, or adaptive transport-map constructions
\cite{ElMoselhy2012,marzouk2017sampling}. In this setting, the intermediate
density $\rho_t$ represents an evolving approximation of the posterior
distribution, while the momentum $\omega_t$ describes transport in parameter space. Additional physical or computational information may naturally constrain the admissible interpolation. For example, one may require the intermediate distributions to remain inside a physically
admissible parameter region, to preserve known moments or quantities of
interest, or to avoid transporting probability mass through regions where the
forward model or surrogate is considered unreliable. Such prior structural information can naturally be incorporated via affine constraints on the transport and source variables.}

{Each of these examples illustrate that affine equality or inequality
constraints over $(\rho,\omega,\zeta)$ can therefore be useful when the goal is
to incorporate mechanism-level information into dynamic unbalanced optimal
transport.}

\subsection{Outline of the paper}
The structure of the paper is as follows. In Section~\ref{sec:background}, we give a brief review of the theoretical perspective on the unbalanced optimal transport problem and the constrained model of \cite{bauer2025path}. Section \ref{sec:generalized_model} introduces the generalized constrained model and includes the main existence results. We present our numerical framework to solve the generalized constrained UOT problem in Section \ref{sec:numerics}. In Section \ref{sec:experiments}, we provide examples of applications to our numerical algorithm to both synthetic and real-world data. The paper concludes with a short discussion of potential future research directions in Section \ref{sec:conclusion}. Additional algorithmic and experimental details are included in the Supplementary Materials.

\section{Background: Dynamic Optimal Transport and the Wasserstein-Fisher-Rao metric}
\label{sec:background}
In this section, we will give a brief review of dynamic (unbalanced) optimal transport and, in particular, of the constrained model introduced in~\cite{bauer2025path}. For more details we refer to~\cite[Section 2 and 3]{bauer2025path}, from which the notation and most of the definitions of this section are taken.

\subsection{Dynamic optimal transport and the Wasserstein distance} From here on let $\Omega \subset \R^n$ denote a fixed compact set with $C^1$ boundary and let $\rho_0$ and $\rho_1$ be probability measures on $\Omega$. To simplify the presentation, we assume that each of the measures $\rho_i$ has density with respect to the Lebesgue measure on $\R^n$. {For the moment, we abuse notation and use  $\rho_i:\Omega \to \R$ to denote the densities.} In this setting, the Benamou-Brenier formulation~\cite{benamou_computational_2000} of optimal transport consists of the optimization problem 
\begin{equation}\label{eqn:dynamic_ot}
\inf_{\rho,\omega} \int_0^1 \int_\Omega \frac{\|\omega(t,x)\|^2}{\rho(t,x)} \; dx \, dt,
\end{equation}
where the infimum is taken over all 
paths of densities $\rho:[0,1] \times \Omega \to \R$ and time-dependent vector fields $\omega:[0,1] \times \Omega \to \R^n$ subject
to the \define{continuity equation}
\begin{equation}\label{eqn:continuity_equation}
\partial_t \rho + \nabla \cdot \omega = 0, \qquad \rho(0,\cdot) = \rho_0, \; \rho(1,\cdot) = \rho_1.
\end{equation}
The seminal results of Benamou and Brenier~\cite{benamou_computational_2000}
state that this formulation of optimal transport is equivalent to the standard Kantorovich formulation and thus the square root of the optimal energy is the so-called \define{Wasserstein distance}. Note that this formulation directly leads to a (formal) Riemannian interpretation of the Wasserstein-distance, where the corresponding energy Lagrangian is given by~\eqref{eqn:dynamic_ot}; the corresponding (formal) Riemannian metric is often called the \textbf{Otto metric}~\cite{OttoPic}. 

\subsection{The Fisher-Rao metric and the Hellinger distance}
Another approach to metrizing the space of (probability) densities is given by the Hellinger distance, which has its origins in the field of information geometry~\cite{amari2016information}. It is the geodesic distance of a Riemannian metric called the Fisher-Rao metric and was originally introduced to compare elements of finite-dimensional submanifolds of probability densities in the context of statistics~\cite{radhakrishna1945information}, but has since then been extended to the infinite dimensional setting; see, e.g.,~\cite{friedrich1991fisher,ay2017information,bauer2016uniqueness}. While the Wasserstein distance (Otto metric, respectively) is defined for probability densities, the Hellinger distance (Fisher-Rao metric, respectively) can be naturally defined on the space of all densities (without a mass preservation constraint).

To define this distance, we restrict ourselves to the space of all measures $\rho$ that have a positive density $\rho:\Omega \to \R_{>0}$ with respect to the Lebesgue measure.
Given two such densities $\rho_0,\rho_1:\Omega \to \R$, the \define{Hellinger distance} between them is given by
\begin{equation}\label{eqn:fisher_rao_distance}
\inf_{\rho,\zeta} \int_0^1 \int_\Omega \frac{\zeta(t,x)^2}{\rho(t,x)} \; dx \, dt, 
\end{equation}
where the infimum is over sufficiently smooth paths of densities $\rho:[0,1] \times \Omega \to \R$ and functions $\zeta:[0,1] \times \Omega \to \R$, 
subject to the \define{continuity equation}
\begin{equation}\label{eq:continuity:Fisher-Rao}
\partial_t \rho = \zeta, \qquad \rho(0,\cdot) = \rho_0, \; \rho(1,\cdot) = \rho_1.
\end{equation}
Note that the energy Lagrangian~\eqref{eqn:fisher_rao_distance} is again quadratic and thus it  implicitly defines a Riemannian metric on the space of positive densities; this is exactly the {\bf Fisher-Rao  metric}. 

\subsection{Unbalanced dynamic optimal transport and the Wasserstein-Fisher-Rao distance}
In the \emph{unbalanced optimal transport} framework~\cite{chizat2018interpolating}, the approaches described above (dynamic optimal transport and Fisher-Rao geometry) are combined to define a new metric on the space of all densities.  The main idea behind this framework is to combine the continuity equations of optimal transport and information geometry to obtain
\begin{equation}\label{eqn:modified_continuity_eqn}
    \partial_t \rho + \nabla \cdot \omega = \zeta,
\end{equation}
where $\zeta:[0,1] \times \Omega \to \R$ is interpreted as a source term, which allows the creation and destruction of mass, and where $\nabla \cdot \omega$ is the transport term.  

When expressed in the form \eqref{eqn:modified_continuity_eqn} (\eqref{eqn:continuity_equation} or \eqref{eq:continuity:Fisher-Rao}, respectively), the continuity equation inherently assumes that the measure $\rho$ admits a differentiable density. As shown in \cite{chizat2018interpolating}, this requirement can be relaxed by instead formulating the equation in its weak (distributional) sense. To introduce the Wasserstein-Fisher-Rao Lagrangian, which serves as the basis of the constrained model studied in this article, we will first recall this notion in the following definition.

\begin{definition}[Distributional Continuity Equation]
Let $\rho_0,\rho_1 \in \Mplus(\Omega)$. We say that a triple $(\rho,\omega,\zeta)\in \Mplus([0,1]\times \Omega)\times \Meas([0,1]\times \Omega)^n\times \Meas([0,1]\times \Omega)$ satisfies the \define{continuity equation in the distributional sense} with boundary conditions $\rho_0$ and $\rho_1$ if, for any $\phi \in C^1([0,1]\times \Omega)$, 
\begin{equation}\label{eqn:dist_continuity_equation}
    \int_{[0,1]\times \Omega} \partial _t \phi d\rho + \int_{[0,1]\times \Omega} \nabla \phi \cdot d\omega  +\int_{[0,1]\times \Omega}\phi  d\zeta = \int_{\Omega}\phi(1,\cdot)d\rho_1 - \int_{\Omega}\phi(0,\cdot)d\rho_0 
\end{equation}
We denote by $\mathcal{CE}(\rho_0,\rho_1)$ the set of triples $(\rho,\omega,\zeta)$ satisfying the continuity equation with boundary conditions $\rho_0,\rho_1$.
\end{definition}

The unbalanced optimal transport distance of \cite{chizat2018interpolating} is then defined as follows. 
\begin{definition}[WFR distance]\label{def:wfr_distance}
    For $\delta>0$, the \define{Wasserstein-Fisher-Rao infinitesimal cost} is the function $f_\delta:\mathbb{R}\times \mathbb{R}^n\times \mathbb{R} \to [0,+\infty]$ defined by 
    \begin{equation}
        f_{\delta}(\rho,\omega,\zeta) \coloneqq \begin{cases}
            \frac{\|\omega\|^2+\delta^2\zeta^2}{2\rho}, & \rho>0 \\ 
            0, & (\rho,\omega,\zeta)=(0,0,0) \\
            +\infty, & \textrm{otherwise.} 
        \end{cases}
    \end{equation}
    The associated \define{Wasserstein-Fisher-Rao (WFR) distance} is then defined by
    \begin{equation}\label{eqn:WFR_distance}
    \mathrm{WFR}_\delta(\rho_0,\rho_1)^2 \coloneqq \inf \left\{ \int_{[0,1]\times \Omega} f_\delta\left(\frac{d\mu}{d\lambda}\right)d \lambda \middle| 
    \begin{aligned}
         \mu = (\rho,\omega,\zeta)\in \mathcal{CE}(\rho_0,\rho_1) 
    \end{aligned}
    \right\},
\end{equation}
where $\lambda$ is any nonnegative measure on $[0,1]\times \Omega$ such that $\rho,\omega,\zeta\ll \lambda$ and $\frac{d\mu}{d\lambda} \coloneqq (\frac{d\rho}{d\lambda}, \frac{d\omega}{d\lambda}, \frac{d\zeta}{d\lambda})$. We note that the value does not depend on the choice of $\lambda$ due to the 1-homogeneity of $f_{\delta}$, cf.~\cite[Definition 2.2]{bauer2025path}.
\end{definition}

\begin{remark}
Considering $\rho$ as an element of $\mathcal{M}^+([0,1] \times \Omega)$, one loses the dynamical interpretation of $\rho$ as a ``path of measures'', but this interpretation is recovered via the concept of disintegration, i.e.,  one can \define{disintegrate} the measure $\nu$ on $[0,1] \times \Omega$ as a Borel family of measures $\{\nu_t\}_{t \in [0,1]}$ such that $d\nu(t,x) = dt \otimes d\nu_t(x)$; cf.~\cite[Proposition~2.4]{bauer2025path}. In the following, we will write in short $\nu = dt \otimes \nu_t$ and we will frequently abuse terminology and refer to $\mu$, $\rho$, $\omega$, and $\zeta$ as \emph{paths of measures}, or simply as \emph{paths}.
\end{remark}

{
\begin{remark}
    We note that the WFR distance was introduced independently in three articles at essentially the same time:~\cite{chizat2018interpolating}, which used the Wasserstein-Fisher-Rao terminology,~\cite{liero2018optimal}, where the terminology \emph{Hellinger-Kantorovich distance} is used, and~\cite{10.57262/ade/1476369298}, that did not introduce any specific name for the distance.  
\end{remark}
}

\subsection{Constrained unbalanced dynamic optimal transport}\label{sec:oldmodel}
Finally, we revisit the constrained, unbalanced optimal transport framework introduced in our earlier work~\cite{bauer2025path}. This framework will then be broadened and generalized in the subsequent sections of this article. 

The idea of the constrained model of~\cite{bauer2025path} was to restrict the set of admissible paths in the dynamic formulation of unbalanced optimal transport to satisfy certain (time dependent) integral constraints on the densities $\rho$. The full model is introduced in the following definition:
\begin{definition}[WFR problem with time-varying constraints]\label{def:wfr_problem_time_varying}
For functions $\eich:[0,1] \times \Omega \to \R^d$ and $\ef:[0,1] \to \R^d$, we say that a measure $\rho \in \Mplus([0,1] \times \Omega)$ \define{satisfies an affine constraint} with respect to $\eich$ and $\ef$ if it disintegrates in time ($\rho = dt \otimes \rho_t$) and 
\begin{equation}
\int_\Omega \Htx d\rho_t(x) = \Ft \qquad \mbox{for a.e.~$t \in [0,1]$}.
\end{equation}
In this case, we write $\rho \in \AHF$. For $\rho_0, \rho_1\in \Mplus(\Omega)$, we consider the constrained set of measures
\begin{equation}\label{eqn:first_constraint_set}
\CEHF\coloneqq \left\{\mu = (\rho,\omega,\zeta) \in \mathcal{CE}(\rho_0,\rho_1) \mid \rho \in \AHF\right\}.
\end{equation}
We define the  \textbf{Wasserstein-Fisher-Rao energy with time-varying affine constraints} by
\begin{equation}
\label{eqn: constWFR_affine}
\WFRAF\rhosq = \inf \left\{ \int_{[0,1]\times \Omega} f_\delta\left(\frac{d\mu}{d\lambda}\right)d \lambda \middle| 
    \begin{aligned}
         \mu = (\rho,\omega,\zeta)\in \CEHF
    \end{aligned}
    \right\},
\end{equation}
 where $\lambda$ is any nonnegative Borel measure on $[0,1]\times \Omega$ such that $\rho,\omega, \zeta \ll \lambda$. By the homogeneity of $f_\delta$, the integral in the definition does not depend on the choice of $\lambda.$ 
\end{definition}
\begin{remark}
If the constraints do not depend on time, i.e., $\eich(t,x)=\eich(x)$ and $\Ft=C>0$, then one can view the constrained Wasserstein-Fisher-Rao energy as an energy Lagrangian on the constrained space of measures 
\begin{equation}
    \MplushC \coloneqq \left\{\rho \in \Mplus(\Omega)\middle| \int_{\Omega} \eich \; d\rho = \Ce\right\},
\end{equation}
and, under additional assumptions on $\eich$, this defines a distance function on this constrained space of measures, cf.~\cite[Theorem~3.8]{bauer2025path}.
\end{remark}

The main result of our previous article~\cite{bauer2025path} established the existence of minimizers to~\eqref{eqn: constWFR_affine} assuming that the problem is feasible, i.e., assuming that there exist finite energy paths that satisfy the constraints, cf.~\cite[Theorem~4.3]{bauer2025path}. In the next section, we will significantly broaden this constrained model by allowing constraints not only on the density path, but also on the paths of controls. The main theoretical result of the present article will then show the existence of minimizers for this extended model.

\section{Generalized Model for Constrained Unbalanced Optimal Transport}
\label{sec:generalized_model}
In this section, we aim to introduce a generalization of the constrained model of~\cite{bauer2025path}, as presented in Section~\ref{sec:oldmodel}. We are specifically interested in
generalizing the allowable constraints in two different directions, while preserving the convexity properties of the resulting variational problems:
\begin{enumerate}
    \item We will consider constraints that operate on the control variables $(\omega,\zeta)$, in addition to the density $\rho$.
    \item We will allow constraints that involve affine inequalities instead of just equalities.
\end{enumerate}
This section formalizes this extended framework and generalizes the main existence result of \cite{bauer2025path} to this setting.

\subsection{Generalized constraint model}

 We begin by introducing a new constrained space of measure triplets 
 \[
 (\rho,\omega,\zeta) \in \Mplus([0,1] \times \Omega)\times \Meas([0,1] \times\Omega)^n\times \Meas([0,1] \times\Omega).
 \]
 
\begin{definition}
Given functions $F:[0,1] \rightarrow \R$, $H^\rho:[0,1] \times \Omega \rightarrow \R$, $H^\omega:[0,1]\times \Omega \rightarrow \R^n$, and $H^\zeta:[0,1]\times \Omega \rightarrow \R$, letting $H=(H^\rho,H^\omega,H^\zeta)$, we say that $\mu=(\rho,\omega,\zeta)$ belongs to the \define{affine constraint space} $\AHFineq(\eich,\ef)$ if the three measures $\rho,\omega$ and $\zeta$ disintegrate in time and
    \[
    \int_{\Omega} H^{\rho}(t,x) d\rho_t(x) + \int_\Omega H^\omega(t,x) \cdot d\omega_t(x) + \int_\Omega H^\zeta(t,x) d\zeta_t(x) \leq F(t) \; \mbox{for a.e.~$t \in [0,1]$}.
    \]  
\end{definition}
We can now introduce a generalized version of the constrained unbalanced optimal transport problem as follows.
\begin{definition}
\label{def:generalized_constrained_WFR}
For $d\in\mathbb N$ consider functions
 $H_i=(H^\rho_i,H^\omega_i,H^\zeta_i):[0,1]\times \Omega \rightarrow \R \times \R^n \times \R$ and $F_i:[0,1]\rightarrow \R$, for $i=1,\ldots,d$. For $\rho_0,\rho_1 \in \Meas^+(\Omega)$, we define the \define{set of constrained paths} from $\rho_0$ to $\rho_1$ by
\begin{align}
\CEHF= \{\mu \in \mathcal{CE}(\rho_0,\rho_1) \ | \ &\mu \in \AHFineq(H_i,F_i) \ \ \text{for } i=1,\ldots,d\}. 
\end{align}
 We then define the \define{Wasserstein-Fisher-Rao energy} with the affine constraints given by $H$ and $F$ similar to \eqref{eqn:WFR_distance}:
\begin{equation}
\label{eq:constrained_WFR_extended}
    \WFRAF(\rho_0,\rho_1)^2 = \inf \left\{ \int_{[0,1]\times \Omega} f_\delta\left(\frac{d\mu}{d\lambda}\right)d \lambda \ \middle| \ 
    \begin{aligned}
         \mu = (\rho,\omega,\zeta)\in \CEHF
    \end{aligned}
    \right\}.
\end{equation}
\end{definition}
In other words, \eqref{eq:constrained_WFR_extended} corresponds to minimizing the Wasserstein-Fisher-Rao cost over paths $\mu$ that satisfy a set of $d$ integral affine  constraints. 
\begin{remark}
{Note that any affine equality constraint can be written as a combination of two affine inequality constraints, and thus fits within the formulation of Definition \ref{def:generalized_constrained_WFR}. With a slight abuse of notation we will sometimes consider such a tuple of two inequality constraints as a single equality constraint.}   
\end{remark}

\begin{remark}
{We also point out that 
we can recover the setting of~\cite{bauer2025path}, as described in Section~\ref{sec:oldmodel}, by restricting to equality constraints and by setting $H^\omega_i=0$, $H^\zeta_i=0$ for all $i=1,\ldots,d$.} On the other hand, setting $H^\rho=0$ leads to constraints that solely operate on the control measures $\omega$ and $\zeta$. 
\end{remark}

\subsection{Existence of minimizers}

We may now generalize the optimal path existence result~\cite[Theorem 4.3]{bauer2025path} to this new setting.

\begin{theorem}\label{thm:main_theorem}
Let $\rho_0,\rho_1 \in \Meas^+(\Omega)$ and $H:[0,1]\times \Omega \rightarrow (\R \times \R^n \times \R)^d$, $F:[0,1]\rightarrow \R^d$ be continuous functions. If the problem \eqref{eq:constrained_WFR_extended} is feasible, i.e., there exists $\mu \in \CEHF$ with finite WFR cost, then an optimal path exists.    
\end{theorem}

\begin{proof}
The proof of this Theorem follows a similar proof strategy as in~\cite{bauer2025path}, based again on the use of the Fenchel-Rockafellar theorem (see~\cite{borwein2005techniques})
Therefore, we start by introducing the following Banach spaces $X$ and $Y$:
    \begin{align}
    X&= C^1([0,1]\times \Omega)\times C([0,1])^d,\\
    Y &= C([0,1]\times \Omega)\times C([0,1]\times \Omega)^n\times C([0,1]\times \Omega).
    \end{align}
We define the convex function ${\mathcal F}: Y\to \mathbb{R}\cup \{+\infty\}$ by
    \[
    {\mathcal F}(\alpha,\beta,\gamma) = \int_{[0,1]\times \Omega}\iota_{B_\delta}(\alpha(t,x),\beta(t,x),\gamma(t,x))dtdx,
    \]
    where $\iota_{B_\delta}:\R \times \R^n \times \R \to \R \cup \{+\infty\}$ is the convex indicator function of the set $B_\delta$, defined by
    \[
        B_{\delta} = \left\{(a,b,c)\in \mathbb{R}\times \mathbb{R}^n\times \mathbb{R} \middle| a+\frac{1}{2}\left(\|b\|^2+\frac{c^2}{\delta^2}\right) \leq 0\right\};
    \]
    that is, $\iota_{B_\delta}(a,b,c) = 0$ if $(a,b,c) \in B_\delta$, and is otherwise equal to $+\infty$. 
    The Fenchel conjugate of $\iota_{B_\delta}$ is the Wasserstein-Fisher-Rao infinitesimal cost $f_\delta$, {see~\cite{chizat2018interpolating} or~\cite[Lemma 5.3]{bauer2025path}.} Next, we define a function $\mathcal{G}:X \rightarrow \R\cup \{+\infty\}$ by
    \begin{align}
        &\mathcal{G}(\phi,\psi)  =  -\sum_{i=1}^{{d}}\int_{0}^{1}\psi_i(t)\ef_i(t)dt  + \iota_{s\preceq 0} \left(\sup_{t \in [0,1]} \psi(t) \right)     +\int_{\Omega}\phi(0,\cdot)d\rho_0 - \int_{\Omega}\phi(1,\cdot)d\rho_1
    \end{align}
    where 
    $\iota_{s \preceq 0}$ is the convex indicator of the nonpositive orthant of $\R^{d}$, i.e., 
        \[
        \iota_{s  \preceq 0}(x_1,\ldots,x_{d}) = \left\{\begin{array}{rl}
        0 & \mbox{if } x_i \leq 0, \ \; \forall \; i = 1,\ldots,d \\
        +\infty &\mbox{otherwise;}
        \end{array}\right.
        \]
        and where  $\sup_{t \in [0,1]} \psi(t) \in \R^{d}$  denotes the coordinate-wise supremum over the interval $[0,1]$. We point out that $\mathcal{G}$ is a convex function over $X$. 
    
    As the last step in setting up the Fenchel-Rockafellar theorem, let $A:X \rightarrow Y$ be the bounded linear operator defined by 
    \begin{equation}
        A(\phi,\psi) = \left(\partial_{t}\phi+\sum_{i=1}^{d}\eich_i^\rho \psi_i, \nabla\phi + \sum_{i=1}^{d}\eich_i^\omega \psi_i, \phi + \sum_{i=1}^{d}\eich_i^\zeta \psi_i\right).
    \end{equation}
    We then consider the primal problem
\begin{equation}\label{eqn:primal_problem}
    \inf_{x \in X} \mathcal{F}(Ax) + \mathcal{G}(x).
    \end{equation}
    Using a similar argument as in the proof of~\cite[Theorem~4.3]{bauer2025path}, one easily checks that there exists $x_0 \in X$ that satisfies the conditions of Fenchel-Rockafellar strong duality theorem---namely, $\mathcal{G}(x_0), \mathcal{F}(Ax_0) < +\infty$ and $\mathcal{F}$ is continuous at $Ax_0$---so that \eqref{eqn:primal_problem} is equivalent to the Lagrange dual problem:
    \begin{equation}\label{eqn:lagrange_dual}
        \sup_{y^* \in Y^*} \, -\mathcal{F}^*(y^*) - \mathcal{G}^*(-A^*y^*).
    \end{equation}
    Using~\cite[Theorem 7.17]{folland2013real}, we can identify the dual $Y^\ast$ as  
    \[
    Y^\ast \approx \Meas([0,1] \times \Omega)\times \Meas([0,1] \times\Omega)^n\times \Meas([0,1] \times\Omega),
    \]
    Following the argument of~\cite{chizat2018interpolating},  $\mathcal{F}^*(y^*)$ is  equal to the objective function of the constrained Wasserstein-Fisher-Rao problem  \eqref{eq:constrained_WFR_extended}. We will show that the second term of \eqref{eqn:lagrange_dual} enforces the constraints of \eqref{eq:constrained_WFR_extended}, which then proves the theorem, as the Fenchel-Rockafellar theorem guarantees the existence of a maximizer in \eqref{eqn:lagrange_dual}, hence of a minimizer for \eqref{eq:constrained_WFR_extended}.
    
    It remains to understand the second term of \eqref{eqn:lagrange_dual}. We have
    \begin{equation}\label{eqn:G_star_expression}
       \mathcal{G}^*(-A^*y^*) = \sup_{x=(\phi,\psi)\in X} \, \{\langle -A^*y^*,x\rangle-\mathcal{G}(x)\} = \sup_{x\in X}\left\{-\langle y^*,Ax\rangle-\mathcal{G}(x)\right\}.
    \end{equation}
    Moreover, we have by definition that
    \begin{align}
       \langle y^*,Ax\rangle = &\int_{[0,1]\times \Omega} \partial_{t}\phi d\rho + \int_{[0,1]\times \Omega}\nabla\phi \cdot d\omega +\int_{[0,1]\times \Omega}\phi d\zeta \\ 
       &\qquad  +\sum_{i=1}^{d}\int_{[0,1]\times \Omega} \psi_i\eich_i^\rho d\rho  + \sum_{i=1}^{d}\int_{[0,1]\times \Omega} \psi_i\eich_i^\omega\cdot d\omega +\sum_{i=1}^{d}\int_{[0,1]\times \Omega} \psi_i\eich_i^\zeta d\zeta,
    \end{align}
    from which we deduce that $\mathcal{G}^*(-A^*y^*)$ can be expressed as the sum of {two} terms:
    \begin{equation}\label{eqn:G_term_1}
        \begin{split}
            &\sup_{\phi} \left\{-\int_{[0,1]\times \Omega} \partial_{t}\phi d\rho - \int_{[0,1]\times \Omega}\nabla\phi \cdot d\omega \right. \\
            &\hspace{1.5in}  \left.-\int_{[0,1]\times \Omega}\phi d\zeta +\int_{\Omega}\phi(1,\cdot)d\rho_1 -\int_{\Omega}\phi(0,\cdot)d\rho_0 \right\},
        \end{split}
    \end{equation}
    {and}
    \begin{equation}\label{eqn:G_term_2}
        \sum_{i=1}^{{d}} \sup_{\psi_i \leq 0}  \int_0^1 \psi_i(t)\Bigg(F_i(t)  - \int_{\Omega} \eich_i^\rho d\rho_t - \int_{\Omega} \eich_i^\omega\cdot d\omega_t - \int_{\Omega} \eich_i^\zeta d\zeta_t\Bigg)dt.
    \end{equation}
    We now describe the properties of each term. 
    The first sup over $\phi$ \eqref{eqn:G_term_1} is zero when $\mu=(\rho,\omega,\zeta)$ satisfies the continuity equation and $+\infty$ otherwise. This term therefore enforces the continuity equation in the supremum \eqref{eqn:lagrange_dual}.

    To get the second term \eqref{eqn:G_term_2}, we have used the presence of the indicator function $\iota_{s\preceq 0} \left(\sup_{t \in [0,1]} \psi(t) \right)$ in $\mathcal{G}$ to deduce that, when taking the supremum in \eqref{eqn:G_star_expression}, one can restrict to functions which satisfy $\psi_i(t) \leq 0$ for a.e.~$t \in [0,1]$ and all $i =1,\ldots,d$. Now suppose that the inequality constraint 
    \[
    \int_{\Omega} \eich_i^\rho d\rho_t + \int_{\Omega} \eich_i^\omega\cdot d\omega_t + \int_{\Omega} \eich_i^\zeta d\zeta_t \leq F_i(t) \quad \mbox{(for a.e.~$t \in [0,1]$)}
    \]
    is satisfied. Then the term in parentheses in \eqref{eqn:G_term_2} is non-negative, hence the supremum over $\psi_i$'s satisfying $\psi_i \leq 0$ is zero. On the other hand, if the inequality constraint fails, then the term in parentheses is negative on a set of positive measure, hence  $\psi_i$'s may be chosen to make \eqref{eqn:G_term_2} arbitrarily large. Putting this together, we deduce that the inequality constraint is enforced by the term \eqref{eqn:G_term_2}.

    Putting all of our observations together, we deduce that 
    \[
    \mathcal{G}^*(-A^*y^*) = \left\{\begin{array}{rl}
0 & \mbox{if } (\rho, \omega, \zeta)\in \CEHF \\
+\infty & \mbox{otherwise.}
\end{array}\right.
    \]
    This finally allows us to recover that the dual problem is equivalent to \eqref{eq:constrained_WFR_extended} and that, by Fenchel-Rockafellar strong duality, the two problems have the same optimal value and that there exists a dual optimal that is therefore a solution of the constrained unbalanced optimal transport problem.
\end{proof}

{
\subsection{Feasibility of the constrained problem}
\label{sec:feasibility}
In order to obtain the existence of minimizers to \eqref{eq:constrained_WFR_extended}, one still needs the problem to be \textit{feasible} in the sense that there exists a finite energy path in $\CEHF$. This is not easy to guarantee in the constrained setting. Within the more restricted model discussed in Section~\ref{sec:oldmodel}, it was shown for several classes of constraints, cf.~\cite[Theorems 3.5, 4.7 and 4.6]{bauer2025path}. In what follows, we will derive a more general feasibility condition, first for the density-only case that extends our previous results in \cite{bauer2025path} and then for constraints involving all control variables. For simplicity, we will restrict to the time-static case (i.e, when the constraint functions $H^{\rho},H^{\omega}$ and $H^{\zeta}$ are independent of time), as sufficient conditions for the time-dependent situation becomes significantly more technical and consequently less insightful.
\begin{theorem}
\label{thm:feasibility_density_constr}
    Let $\rho_0,\rho_1\in \Mplus(\Omega)$ 
    and let $\mu=(\rho,\omega,\zeta)$ be given by:
    \begin{equation}
        \rho_t = (1-\theta(t))\rho_0+\theta(t)\rho_1,\quad \omega_t=0, \quad \zeta_t = \theta'(t)(\rho_1-\rho_0),
    \end{equation}
    where $\theta(t) = \sin^{2}{\left(\tfrac{\pi t}{2}\right)}$.
    Then, $\mu \in \CE$ with WFR energy bounded from above by 
    \begin{equation}
     \int_{[0,1]\times \Omega}f_\delta\left(\frac{d\mu}{d\lambda}\right)d\lambda \leq \frac{\pi^2\delta^2}{4}(\rho_0(\Omega)+\rho_1(\Omega))
    \end{equation}
    Furthermore, if $\rho_0$ and $\rho_1$ satisfy an affine density constraint then $\rho_t$ satisfies the same constraint for all time $t$. Thus every collection of stationary affine density equality and inequality constraints is feasible whenever both endpoints $\rho_0$ and $\rho_1$ satisfy those constraints.
\end{theorem}
\begin{proof}
    Direct substitution of $(\rho,\omega,\zeta)$ into the distributional continuity equation and integration by parts show that $(\rho,\omega,\zeta)\in \CE$. We will now show the energy bound. Let $\lambda =dt\otimes (\rho_0+\rho_1)$ and let $f_i,i =0,1$ be the density of $\rho_i$ with respect to $\rho_0+\rho_1$. 
    Then the integrand of the WFR energy becomes
    \begin{equation}
    f_\delta\left(\frac{d\mu}{d\lambda}\right) = \frac{\delta^2}{2}\frac{|\theta'(t)|^2(f_1-f_0)^2}{(1-\theta(t))f_0+\theta(t)f_1} 
    \end{equation}
    whenever the denominator is positive (and otherwise zero). For brevity we will suppress from here on the dependance of $\theta$ on $t$. Using the non-negativity of $f_0,f_1$, and $\theta$, we can bound the right hand side via:
    \begin{align}
        \frac{(f_1-f_0)^2}{(1-\theta)f_0+\theta f_1}  &\leq \frac{f_0}{1-\theta} + \frac{f_1}\theta.
    \end{align}
Using in addition the following two elementary integral formulae 
\[
    \int_{0}^{1}\frac{|\theta'|^2}{\theta}dt = \int_{0}^{1}\pi^2\cos^2{\left(\frac{\pi t}{2}\right)}dt=\frac{\pi^2}{2}, \quad \int_{0}^{1}\frac{|\theta'|^2}{1-\theta}dt = \int_{0}^{1}\pi^2\sin^2{\left(\frac{\pi t}{2}\right)}dt=\frac{\pi^2}{2},
\]
the WFR energy of $\mu$ can then be bounded as follows:
\begin{align}
    \int_{[0,1]\times \Omega}f_\delta\left(\frac{d\mu}{d\lambda}\right)d\lambda &\leq \delta^2\int_{\Omega}\int_{[0,1]}\left(\frac{|\theta'|^2f_0}{1-\theta}+\frac{|\theta'|^2f_1}{\theta}dtd(\rho_0+\rho_1)\right) \\ &=\frac{\delta^2}{2}\rho_0(\Omega)\int_{0}^1 \frac{|\theta'|^2}{1-\theta}dt + \frac{\delta^2}{2}\rho_1(\Omega)\int_{0}^1 \frac{|\theta'|^2}{\theta}dt \\ &=\frac{\delta^2}{2}\rho_0(\Omega)\frac{\pi^2}{2} + \frac{\delta^2}{2}\rho_1(\Omega)\frac{\pi^2}{2}.
\end{align}
The statement that $\rho_t$ satisfies the same affine constraints as $\rho_0$ and $\rho_1$ is immediate.  
\end{proof}

It follows from Theorems \ref{thm:main_theorem} and \ref{thm:feasibility_density_constr} that for any continuous $H^\rho:\Omega\to\R^d$ and $C\in\R^d$, the WFR problem \eqref{eq:constrained_WFR_extended} with static constraints on the density $\rho$ is well-posed provided that the two boundary measures $\rho_0$ and $\rho_1$ satisfy those constraints. Thereby the results of the present article significantly strengthen the corresponding result in our previous work~\cite{bauer2025path}. Let us now examine the case of constraints that further involve the control variables $\omega$ and $\zeta$. 

\begin{theorem}
\label{thm:feasibility_controls_constr}
   Let $(H_i^\rho,H_i^\omega,H_i^\zeta)_{i=1,\ldots,d}$ be functions independent of $t$, continuous on $\Omega$ with $H_i^\zeta \in C^1(\Omega)$, and $C \in \mathbb{R}^d$. Assume that there exist a compact subset $U \subset \text{int}(\Omega)$ with nonempty interior such that 
   \begin{enumerate}
       \item for all i, $H^\omega_i,H_i^\zeta \in C^2(U)$;
       \item the Gram matrix $G=(G_{ij}) \in \mathbb{R}^{d \times d}$ given by:
   \begin{equation}
        G_{ij} = \int_U K_i(x) \cdot K_j(x) dx;\qquad K_i = H^\omega_i - \nabla H_i^\zeta \in C^1(U)
        \end{equation}
        is positive definite
        \item there exists $m>0$ such that for any Borel set $B \subseteq U$ one has $\rho_0(B),\rho_1(B) \geq m|B|$.
        \end{enumerate}
  Then problem \eqref{eq:constrained_WFR_extended} with the static constraints associated to $(H_i^\rho,H_i^\omega,H_i^\zeta)$ and $C$ is feasible as soon as $\rho_0$ and $\rho_1$ both satisfy those constraints.
\end{theorem}
\begin{proof}
    Let us first define a reference path $\bar{\mu} = (\bar{\rho},\bar{\omega},\bar{\zeta})$ as in Theorem \ref{thm:feasibility_density_constr}, i.e. $\bar{\rho}_t = (1-\theta(t))\rho_0 + \theta(t) \rho_1$, $\bar{\omega}_t = 0$ and $\bar{\zeta}_t = \theta'(t)(\rho_1-\rho_0)$. By Theorem \ref{thm:feasibility_density_constr}, $\bar{\mu} \in \CE$, has finite WFR energy and, given the assumption on $\rho_0,\rho_1$, is such that for all $t\in[0,1]$ and Borel set $B \subseteq U$, $\bar{\rho}_t(B) \geq m |B|$. We then define the function $c:[0,1] \rightarrow \R^d$ by 
    \begin{equation}
    \label{eq:def_ci}
        c_i(t) = \int_\Omega H_i^\rho(x) d\bar{\rho}_t(x) + \int_\Omega H_i^\zeta(x) d\bar{\zeta}_t(x), \qquad i=1,\ldots,d,
    \end{equation}
    representing the values of the constraints for the reference path $\bar{\mu}$ (as $\bar{\omega}_t = 0$). Now, the idea is to set $\rho=\bar{\rho}$ and modify $\bar{\omega}$ and $\bar{\zeta}$ in a way that the affine constraints become satisfied while still preserving the continuity equation. Specifically, we introduce the vector fields $v_t \in C^1(U)$ and the measures $\widetilde{\omega} \in \Meas([0,1]\times \Omega)^n$, $\widetilde{\zeta} \in \Meas([0,1]\times \Omega)$:
    \begin{equation}
       v_t(x) = \sum_{j=1}^{d} [G^{-1}(C-c(t))]_j K_j(x), \ \  \widetilde{\omega}_t = v_t dx, \ \ \widetilde{\zeta}_t = \nabla \cdot v_t dx.
    \end{equation}
    Let us then define $\rho = \bar{\rho}$, $\omega = \bar{\omega} + \widetilde{\omega}$ and $\zeta = \bar{\zeta} + \widetilde{\zeta}$. 
    
    We first see that $\mu = (\rho,\omega,\zeta) \in \CE$. Indeed, we have for all $\phi \in C^1([0,1]\times \Omega)$ that
    \begin{equation}
        \int_0^1 \int_\Omega \nabla \phi \cdot d\widetilde{\omega}_t dt + \int_0^1 \int_\Omega \phi \, d\widetilde{\zeta}_t dt = \int_0^1 \int_\Omega \nabla \phi \cdot v_t dx - \int_0^1 \int_\Omega \nabla \phi \cdot v_t \, dx = 0,
    \end{equation}
    where we used the divergence formula for the second term and the fact that $v_t$ vanishes outside of $U$ and thus on $\partial \Omega$. Thus, $(0,\widetilde{\omega},\widetilde{\zeta})$ satisfies the continuity equation, from which we conclude that $\mu \in \CE$. 
    
    Second, we check that $\mu$ verifies all the affine constraints (with equality). For any $i=1,\ldots,d$ and $t \in [0,1]$, one has:
    \begin{align}
        &\int_\Omega H_i^\omega \cdot d\widetilde{\omega}_t + \int_\Omega H_i^\zeta \, d\widetilde{\zeta}_t = \int_\Omega H_i^\omega \cdot v_t dx + \int_\Omega H_i^\zeta \, (\nabla \cdot v_t) \, dx 
        = \int_\Omega \underbrace{(H_i^\omega - \nabla H_i^\zeta)}_{=K_i} \cdot v_t \, dx \\
        &\qquad= \int_U \sum_{j=1}^{d} [G^{-1}(C-c(t))]_j K_i \cdot K_j \, dx = \sum_{j=1}^{d} G_{ij} [G^{-1}(C-c(t))]_j = C_i - c_i(t).
    \end{align}
    From this and \eqref{eq:def_ci} it follows that \[\int_\Omega H_i^\rho d\rho_t + \int_\Omega H_i^\omega \cdot d\omega_t + \int_\Omega H_i^\zeta d\zeta_t = c_i(t) + C_i-c_i(t) = C_i.\]
    
    Lastly, we need to guarantee that the WFR energy of $\mu$ stays finite. We first note that the function $f_\delta$ from Definition \ref{def:wfr_distance} is sub-additive as it is both convex and 1-homogeneous. Denoting the energy $D_\delta(\mu) = \int_0^1 \int_\Omega f_\delta(d\mu/d\lambda) d\lambda$, it follows that: 
    \begin{equation}
    \label{eq:bound_ener_mu}
        D_\delta(\mu) \leq D_\delta(\bar{\rho}/2,\bar{\omega},\bar{\zeta}) + D_\delta(\bar{\rho}/2,\widetilde{\omega},\widetilde{\zeta}) = 2 D_\delta(\bar{\rho},\bar{\omega},\bar{\zeta}) + 2 D_\delta(\bar{\rho},\widetilde{\omega},\widetilde{\zeta}). 
    \end{equation}
    The first term on the right hand side is finite since $D_\delta(\bar\mu) < +\infty$. As for the second term, from the definition of $v_t$ and using the positive-definiteness assumption on $G$ together with the continuity of the $K_j$'s on the compact set $U$, there exists a constant $\beta_1>0$ such that $\|v_t\|^2_{L^2(U)} \leq \beta_1 \|C-c(t)\|^2$. Similarly, as the $K_j$'s are $C^1$ on $U$, one has $\nabla \cdot K_j$ continuous and thus, one can again find a constant $\beta_2 >0$ for which $\|\nabla \cdot v_t\|^2_{L^2(U)} \leq \beta_2 |C-c(t)|^2$. From the definitions of $\widetilde{\omega}_t$, $\widetilde{\zeta}_t$ and the fact that for all $t\in[0,1]$, $\bar{\rho}_t \geq m \, dx$, we have $\widetilde{\omega}_t,\widetilde{\zeta}_t \ll \bar{\rho}_t$. We then get, on the one hand:
    \begin{equation}
        \int_\Omega \left\|\frac{d\widetilde{\omega}_t}{d\bar{\rho}_t} \right\|^2 d\bar{\rho}_t = \int_U \left\|\frac{d\widetilde{\omega}_t}{d\bar{\rho}_t} \right\|^2 d\bar{\rho}_t \leq \frac{1}{m} \int_U \|v_t\|^2 dx \leq \frac{\beta_1}{m} \|C-c(t)\|^2
    \end{equation}
    and on the other hand:
    \begin{equation}
        \int_\Omega \left|\frac{d\widetilde{\zeta}_t}{d\bar{\rho}_t} \right|^2 d\bar{\rho}_t = \int_U \left|\frac{d\widetilde{\zeta}_t}{d\bar{\rho}_t} \right|^2 d\bar{\rho}_t \leq \frac{1}{m} \int_U |\nabla \cdot v_t|^2 dx \leq \frac{\beta_2}{m} \|C-c(t)\|^2.
    \end{equation} 
    Those two bounds in turn imply that $D_\delta(\bar{\rho},\widetilde{\omega},\widetilde{\zeta})< +\infty$ and thus, by \eqref{eq:bound_ener_mu}, that $\mu$ has finite WFR energy. 
\end{proof}

Once again, using both Theorem \ref{thm:main_theorem} and Theorem \ref{thm:feasibility_controls_constr} one recovers the well-posedness of the WFR problem \eqref{eq:constrained_WFR_extended} under the above assumptions. We note that for a single static constraint, i.e. $d=1$, the positive-definiteness assumption just reduces to the mild requirement that $H^\omega \not\equiv \nabla H^\zeta$ on $U$. Additionally, condition 3 is in particular satisfied as soon as the two measures $\rho_0$ and $\rho_1$ have continuous positive densities with respect to the Lebesgue measure on some common compact subset $U$ with nonempty interior, which is always the case, e.g., for two Gaussian measures.
}
\begin{remark}
\label{rem:feasibility-experiments}
For the specific examples presented in Section~\ref{sec:experiments}, feasibility always holds. Indeed, in the experiments of Sections~\ref{sec:experiments:mass}, \ref{subsec:barrier_constraint},~\ref{subsec:convex_curve_constraint} and~\ref{ssec:experiments:population}, this is already guaranteed by the results of~\cite[Theorems~4.6 and 4.7]{bauer2025path}. For the example of Section~\ref{ssec:experiments:flow}, it also holds as the constraint is only on the momentum variable $\omega$ and thus a pure Fisher-Rao geodesic (with no transport) gives rise to a finite energy path satisfying the constraints. Lastly, for the example studied in Section~\ref{ssec:experiments:budget}, which involves two constraints on the source term $\zeta$, the feasibility can be recovered based on Theorem \ref{thm:feasibility_controls_constr} since $\rho_0,\rho_1$ are both Gaussian densities in that case. %
\end{remark}

\section{Discrete Problem and Algorithmic Approach}\label{sec:numerics}
In order to derive a numerical approach to tackle the constrained UOT model of Section \ref{sec:generalized_model}, with provable convergence guarantees, we shall adapt the algorithm of \cite{chizat2018interpolating} in the unconstrained setting, which is itself based on the Douglas-Rachford (DR) splitting scheme \cite{douglas1956numerical,glowinski1975approximation}. In the following sections, we will describe the chosen discretization in more detail, as well as our proposed numerical framework; pseudocode for the algorithm can be found in the supplementary material~\ref{appendix:algo} and an open source implementation is available at our GitHub repository:  \url{https://github.com/mao1756/acuot}.

\subsection{Discretized problem}
Throughout this section, we consider the time domain to be $[0,1]$ and the space domain $\Omega$ to be \(\Omega=[0, L]\) where $L$ is the length of the domain in the spatial dimension. We consider the $0$th dimension as the time dimension and the $1$st dimension as the spatial dimension. Furthermore, letting \(N_0, N_1\) be the number of discretization steps in time and space, respectively, we shall use uniform discretization steps of the form \(h_0=1/N_0\) and \(h_1=L/N_1\).

We first consider the discretization of the domain. Rather than discretizing all continuous variables $(\rho,\omega,\zeta)$ of the variational problem on a single time-space grid, we will consider, following the idea first proposed in \cite{papadakis2014optimal}, two types of grids: namely, a centered grid and staggered grids, where staggered grids are shifted by half a cell in a particular dimension. Specifically, the centered grid \(\mathcal{G}_c\) is defined as
\[
\mathcal{G}_c=\Big\{(t_{j_0},x_{j_1}) :\ t_{j_0}=\tfrac{j_0+\tfrac12}{N_0},\ \ 
x_{j_1}= \tfrac{(j_1+\tfrac12)L}{N_1}\Big\}
\]
while the staggered grids for time and space dimensions are then given respectively by:
\begin{equation}
\begin{aligned}
&\mathcal{G}_s^t
= \Big\{(t_{j_0},x_{j_1}) :\ t_{j_0} = \tfrac{j_0}{N_0},\ 
x_{j_1} = \tfrac{(j_1+\frac12)L}{N_1},\ \ j_0\in\llbracket 0,N_0\rrbracket,\ 
j_1\in\llbracket 0,N_1-1\rrbracket\Big\},\\
&\mathcal{G}_s^{x} = \Big\{(t_{j_0},x_{j_1}) :\ t_{j_0}=\tfrac{j_0+\tfrac12}{N_0},\
x_{j_1}=\tfrac{j_1 L}{N_1},\ j_0 \in \llbracket 0, N_0 - 1\rrbracket,\ j_1\in\llbracket 0,N_1\rrbracket\Big\}.
\end{aligned}
\end{equation}
We shall then consider centered variables $V=(\rho,\omega,\zeta) \in \mathcal{E}_c\coloneqq \mathbb{R}^{\mathcal{G}_c}\times \mathbb{R}^{\mathcal{G}_c} \times \mathbb{R}^{\mathcal{G}_c}$ where each of the three components is defined on the centered grid, as well as staggered variables $U=(\bar\rho,\bar\omega,\bar\zeta)\in \mathcal{E}_s\coloneqq\mathbb{R}^{\mathcal{G}_s^t}\times \mathbb{R}^{\mathcal{G}_s^{x}}\times \mathbb{R}^{\mathcal{G}_c}$ where the density and momentum field are sampled on the corresponding staggered grids. With a slight abuse of notations, we will write $\rho_{j_0,j_1},\omega_{j_0,j_1},\zeta_{j_0,j_1}$ to denote the values of $\rho,\omega$ and $\zeta$ sampled at $(t_{j_0},x_{j_1}) \in \mathcal{G}_c$ and similarly for the staggered variables $\bar\rho,\bar\omega,\bar\zeta$. The idea behind the use of these different types of grids and variables is that it will provide a more natural way to evaluate the discrete derivatives involved in the continuity equation, as explained in the next subsection. Note that the source term is always defined on the centered grid since no derivative of $\zeta$ is involved in the problem formulation. 

One can convert a staggered variable to a centered variable via the interpolation operator \(I:\mathcal{E}_s \to \mathcal{E}_c\) which is defined as $I(U) = (I_\rho\bar\rho, I_\omega\bar\omega,\bar\zeta)$ where
\begin{align}
(I_\rho\bar\rho)_{j_0,j_1} &= \tfrac12\big(\bar\rho_{j_0,j_1}+\bar\rho_{j_0+1,j_1}\big),\\
(I_{\omega}\bar\omega)_{j_0, j_1} &= \tfrac12\big(\bar\omega_{j_0,j_1}+\bar\omega_{j_0,j_1+1}\big)%
\end{align}
This allows us to then define the consistency condition between the staggered and centered representations as \(V=I(U)\). We note that the above definitions can be easily modified for periodic boundary conditions (in space), by simple cyclic wrapping of the indices.

Based on the previous definitions, we now examine the discretization of the continuity equation $\partial_t \rho + \textrm{div }\omega = \zeta$. First, we define the time-space divergence operator \(\textrm{div}\). For a staggered variable \(U=(\bar\rho,\bar\omega, \bar{\zeta})\), \(\textrm{div}: \mathcal{E}_s \to \mathbb{R}^{\mathcal{G}_c}\) is defined by forward face-to-cell differences:
\begin{align}
(\textrm{div}\,U)_{j_0, j_1}
&= \frac{\bar\rho_{j_0+1,j_1}-\bar\rho_{j_0,j_1}}{h_0}
 +\frac{\bar\omega_{j_0,j_1+1}-\bar\omega_{j_0,j_1}}{h_1}.
\end{align}
The discrete continuity equation can then be written as \((\textrm{div} - s_z)(U) =0\) where \(s_z\) is the auxiliary operator that extracts the source term from a staggered variable, i.e. \(s_z(U) = \bar{\zeta}\). In addition, one also needs to enforce the initial and terminal conditions on the density as well as the no-flux (Neumann) spatial boundary conditions on the momentum. To this end, we introduce a boundary operator $b:\mathcal{E}_s\to \mathbb{R}^{2N_0+2N_1}$ that extracts the corresponding boundary values of \(\bar\rho\) and \(\bar\omega\):
\begin{equation}
  b(U) = \begin{pmatrix}
    \bar\rho_{0,j_1} & \bar\rho_{N_0,j_1} & \bar\omega_{j_0,0} & \bar\omega_{j_0,N_1}
  \end{pmatrix}
\end{equation}
and express the boundary conditions as \(b(U)=b_0\) where \(b_0 = (\rho_0, \rho_1, 0, 0)\), with $\rho_0,\rho_1$ being the initial and final discrete densities on the centered grid. Note that the above corresponds to imposing Neumann boundary conditions in space. As already mentioned previously, this can be easily adapted to periodic spatial boundary conditions by removing the no-flux constraints and instead applying cyclic index wrapping.  

We are now just left with the discrete version of the WFR energy and the path constraints. Based on Definition \ref{def:wfr_distance}, for a centered variable \(V=(\rho,\omega,\zeta)\), the contribution of each time-space cell \(j=(j_0,j_1)\) to the whole WFR cost can be simply taken as $f_\delta(\rho_j,\omega_j,\zeta_j)$. The full cost function is then computed as the discretized time-space integral, i.e.
\begin{equation} \label{eq:discrete_WFR_cost}
    J(V) \coloneqq \sum_{j\in\mathcal{G}_c} f_\delta(\rho_j,\omega_j,\zeta_j)h_0 h_1
\end{equation}
We note that we use the scaling property \cite[Proposition 3.1]{chizat2018interpolating} to reduce the problem to the case $\delta=1$ if we are given values of $\delta$ other than $1$.  Now, when it comes to the additional path constraints on $\rho,\omega,\zeta$, these will be defined on the centered grid via the discretized functions $H^{\rho}_i,H^{\omega}_i,H^{\zeta}_i \in \mathbb{R}^{\Gc}$ for $i=1,\ldots,d$, with $d$ being again the total number of constraints. Slightly departing from the continuous formulation of Section \ref{sec:generalized_model}, we represent every constraint  at each time as a box constraint given by the lower and upper bound variables \(\ell_i \in \overline{\mathbb{R}}^{N_0}\) and $u_i \in \overline{\mathbb{R}}^{N_0}$ with \(\ell_i\le u_i\) component wise. We can then introduce the following constraint operators and constraint sets:
\begin{align}
\mathcal{H}_i(t_{j_0};V)
&\coloneqq\big\langle (H^{\rho}_i)_{j_0,\cdot},\rho_{j_0,\cdot}\big\rangle
 +  \big\langle (H^{\omega}_i)_{j_0,\cdot},\omega_{j_0,\cdot}\big\rangle
 + \big\langle (H^{\zeta}_i)_{j_0,\cdot},\zeta_{j_0,\cdot}\big\rangle,
\\
\mathcal{C}_i&\coloneqq\{V:\ \ell_{i,j_0}\le \mathcal{H}_i(t_{j_0};V)\le u_{i,j_0}\ \forall j_0\in \llbracket0, N_0 - 1\rrbracket\}.
\end{align}
where the bracket notation stands for the discretized $L^2$ inner product on $[0,L]$, for which we use the simple numerical integration scheme $\langle v,w \rangle = \sum_{j_1=0}^{N_1-1} v_{j_1} w_{j_1} h_1$. Note that since we allow the components of $\ell_i$ and $u_i$ to take infinite values and be equal to one another, this setting does indeed provide a discrete equivalent of the integral inequality/equality constrained framework developed in Section \ref{sec:generalized_model}.

We finally arrive at the discretized constrained UOT problem, which we formulate as the following optimization program over a pair $(U,V)$ of centered and staggered variables:  
\begin{equation}\label{eq:disc-constr-wfr}
\min_{U,V}\ J(V)\,
\quad\text{s.t.}\quad
(\textrm{div}-s_z)(U) = 0,\ \ b(U)= b_0,\ \ V=I(U),\ \ V\in \bigcap_{i=1}^d \mathcal{C}_i.
\end{equation}
Using convex indicator functions for all the constraints, we may rewrite the above as
\begin{equation}\label{eq:optimization_target}
\min_{U,V}\ J(V)\;+\;\iota_{\mathcal{CE}}(U)\;+\;\iota_{\{V=I(U)\}}(U,V)\;+\;\sum_{i=1}^d \iota_{\mathcal{C}_i}(V),
\end{equation}
where \(\mathcal{CE}=\{U:(\textrm{div}-s_z)(U) = 0, b(U)= b_0\}\). We note that this is a convex problem. Similar to the continuous setting, the existence of minimizers can be guaranteed under the following basic assumption:
\begin{theorem}
  \label{thm:existence_discrete_wfr}
  There exists a minimizer to the discrete constrained WFR problem (\ref{eq:optimization_target}) provided that the feasible set is nonempty. Moreover, this condition can be checked by solving a second-order cone feasibility problem.
\end{theorem}
Similar to the continuous setting and Theorem \ref{thm:main_theorem}, this result ensures the existence of a solution provided that one can find a discrete path with finite energy \eqref{eq:discrete_WFR_cost} satisfying all constraints, which itself can be verified by solving a second-order cone feasibility problem. Our proof will rely on the following intermediate result of convex analysis. For standard results and definitions, we refer to \cite{RockafellarWets1998}.

\begin{proposition}
  \label{prop:horizon_cone_affine}
  Let $V, W$ be finite dimensional normed vector spaces, $L: V\to W$ be a linear map and $w\in W$. Let $C = \{v\in V: Lv = w\}$ be an affine subspace of $V$. Then, the horizon cone \cite[Definition 3.3]{RockafellarWets1998} of $C$ is the kernel of $L$ i.e. $C^{\infty} = \ker(L)$.
\end{proposition}
\begin{proof}
  Since $C$ is convex and closed, by Theorem 3.6 of \cite{RockafellarWets1998}, we have $C^{\infty} = \{y\in V: \exists x\in C, \forall t\geq 0, x+ty \in C\}$. Let $y\in C^{\infty}$, then there exists $x\in C$ and for all $t\geq 0$, we have $L(x+ty) = w + tLy = w$. This implies that $Ly=0$ i.e. $y\in \ker(L)$. Conversely, if $y\in \ker(L)$, then for any $x\in C$ and $t\geq 0$, we have $L(x+ty) = Lx + tLy = w + 0 = w$. This implies that $x+ty \in C$ and thus $y\in C^{\infty}$. Therefore, we have shown that $C^{\infty} = \ker(L)$.
\end{proof}

We now proceed to the proof of Theorem \ref{thm:existence_discrete_wfr}:
\begin{proof}[Proof of Theorem~\ref{thm:existence_discrete_wfr}]
Let us first rewrite the objective function:
\begin{equation}
    F(U, V) = J(V) + \iota_{\mathcal{CE}}(U) + \iota_{\{V=I(U)\}}(U,V) + \sum_{i=1}^d \iota_{\mathcal{C}_i}(V)
\end{equation}
The main element of the proof is to show that the horizon function \cite[Definition 3.17]{RockafellarWets1998} of $F$, denoted $F^\infty$ in the following, is zero only when $U = 0, V = 0$. This will imply that $F$ is level-bounded by \cite[Theorem 3.26(a), Corollary 3.27]{RockafellarWets1998} since $F$ is proper, convex and lower semicontinuous. (Note that properness of $F$ i.e. that $F(U, V) <\infty$ for some $U, V$, follows from the feasibility assumption.) The existence of minimizers then follows from the fact that a lower-semicontinuous proper and level-bounded function attains a finite minimum \cite[Theorem 1.9]{RockafellarWets1998}. By the sum property of the horizon functions \cite[Exercise 3.29]{RockafellarWets1998}, and the fact that $\iota_{C}^{\infty}= \iota_{C^\infty}$, we have:
\begin{align}
F^{\infty} &\geq J^{\infty} + \iota_{\mathcal{CE}}^{\infty} + \iota_{\{V=I(U)\}}^{\infty} + \sum_{i=1}^d \iota_{\mathcal{C}_i}^{\infty} \\ &\geq J^{\infty} + \iota_{\mathcal{CE}^{\infty}} + \iota_{\{V=I(U)\}^\infty}
\end{align}
Note that $J^{\infty} = J$ since $f_\delta$ is positively homogeneous. Now, if $F^{\infty}(U, V)= 0$, then we must have $J(V) = 0, (U, V)\in \mathcal{CE}^{\infty}, (U, V) \in \{(U,V)|V=I(U)\}^{\infty}$. If we write $V = (\rho, \omega, \zeta)$, the definition of $f_\delta$ implies that $J(V)=0$ only when $\omega = \zeta = 0$. Now, by Proposition \ref{prop:horizon_cone_affine}, we have
\begin{align}
    &\mathcal{CE}^{\infty} = \{(\rho,\omega,\zeta): (\textrm{div} - s_z)(U) = 0, b(U) = 0\} \\
    &\{(U,V)|V=I(U)\}^{\infty} = \{(U,V): V = I(U)\}
\end{align}
If we write $U = (\bar{\rho}, \bar{\omega}, \bar{\zeta})$, the condition $(U,V)\in \{(U,V)|V=I(U)\}$ implies that $\zeta = 0$ too. Moreover, we must have $\bar{\omega}_{j_0, j_1 + 1} = -\bar{\omega}_{j_0, j_1}$ for all $j_0 \in \llbracket 0, N_0 -1 \rrbracket, j_1 \in \llbracket 0, N_1 - 1 \rrbracket$. On the other hand, since we must have $b(U) = 0$, we have $\bar{\omega}_{j_0, 0} = \bar{\omega}_{j_0, N_1} = 0$ for all $j_0 \in \llbracket 0, N_0 -1 \rrbracket$. This implies that $\bar{\omega} = 0$. Finally, the condition $(\textrm{div} - s_z)(U) = 0$, together with $\bar{\omega} = 0, \bar{\omega} = 0$, implies that
\begin{equation}
    \frac{\bar{\rho}_{j_0 + 1, j_1} - \bar{\rho}_{j_0, j_1}}{h_0} = 0 \textrm{ for all } j_0 \in \llbracket 0, N_0 -1 \rrbracket, j_1 \in \llbracket 0, N_1 - 1 \rrbracket
\end{equation}
Together with the boundary condition $b(U) = 0$, this implies that $\bar{\rho} = 0$. $V=I(U)$ then implies that $\rho=0$. Overall, we have $U = 0, V= 0$. Therefore, we have shown that $F^{\infty}(U, V) = 0$ only when $U = 0, V = 0$. This concludes the proof of solution's existence.

Now, regarding the feasibility condition, we see that $(U,V)$ belongs to the feasible set of problem \eqref{eq:disc-constr-wfr} if and only if it satisfies the set of conditions $(\textrm{div} - s_z)(U) = 0, \ b(U) = b_0, \ V = I(U)$ and $V \in \cap_{i=1}^d \mathcal{C}_i$, which are all affine equality and inequality constraints, and the condition that the energy $J(V) < +\infty$. The latter translates to $\rho_j >0$ or $(\rho_j,\omega_j,\zeta_j) = (0,0,0)$ for any $j=(j_0,j_1) \in \mathcal{G}_c$. This is equivalent to the existence of $\tau \geq 0$ such that for all $j$, $\|(\omega_j,\delta \zeta_j)\|_2^2 \leq 2\tau \rho_j$. Indeed, one can take $t$ to be the largest of the $\frac{\|(\omega_j,\delta \zeta_j)\|_2^2}{2\rho_j}$ over all $j$ such that $\rho_j>0$. We can rewrite each of these conditions as follows:
\begin{equation}\label{eq:second_order_cone}
\begin{aligned}
    &\|(\omega_j,\delta \zeta_j)\|_2^2 \leq 2\tau \rho_j = (\tau+\rho_j)^2 -\tau^2 -\rho_j^2 \ \Leftrightarrow \|(\omega_j,\delta \zeta_j)\|_2^2 +\tau^2 + \rho_j^2 \leq (\tau+\rho_j)^2 \\
    &\Leftrightarrow \|(\rho_j,\omega_j,\delta \zeta_j,\tau)\|_2^2\leq (\tau+\rho_j)^2 
\end{aligned}
\end{equation}
which is in turn equivalent to $\|(\rho_j,\omega_j,\delta \zeta_j,\tau)\|_2 \leq \tau + \rho_j$. Thus, the feasibility set of the problem can be expressed as an intersection of hyperplanes, halfspaces and second-order cone conditions on $(\rho,\omega,\zeta,\tau)$.  
\end{proof}

\subsection{Algorithmic approach}
\label{ssec:ppxa_algo}
The Douglas-Rachford splitting algorithm is a generic approach for optimizing the sum of two convex, potentially non-smooth functions. In the context of the present work, having multiple constraints to deal with, we will instead rely on the parallel proximal algorithm (PPXA) \cite{combettes-pesquet-ppxa-2008,combettes-pesquet-signal-2011} which extends the Douglas-Rachford scheme to optimization problems involving the sum of multiple functions. The core idea of PPXA is to rewrite the composite objective function
\begin{equation}
  f_1(x) + f_2(x) + \cdots + f_m(x)
\end{equation}
as $F(x_1, \cdots, x_m) + \iota_{C}(x_1,\cdots ,x_m)$ where $F(x_1, \cdots, x_m) = f_1(x_1) + f_2(x_2) + \cdots + f_m(x_m)$, $C = \{(x_1, \cdots, x_m) | x_1 = x_2 = \cdots = x_m\}$, and $\iota_C$ is the convex indicator function for this set (as we used in the proof of Theorem \ref{thm:main_theorem}). Thus, one can view the objective function as a sum of two convex functions $F$ and $\iota_{C}$, and then apply the Douglas-Rachford splitting to this reformulated problem. This approach is called the parallel proximal algorithm (PPXA) \cite{combettes-pesquet-ppxa-2008,combettes-pesquet-signal-2011} since the proximal operators of each $f_i$ can be computed in parallel.

We will now describe our main PPXA iteration. Letting \(x=(U,V)\) be the current iterate and \(y_k,\pi_k\) auxiliary copies,
we group \(J+\iota_{\mathcal{CE}}\) as block \(f_1\), the consistency constraint \(\{V=I(U)\}\) as block \(f_2\), and the \(d\) box constraints as the blocks \(f_{3},\dots,f_{d+2}\). Then a PPXA step reads
\begin{align}
&\pi_1=\textrm{prox}_{\gamma(J+\iota_{\mathcal{CE}})}(y_1),\qquad
\pi_2=\textrm{prox}_{\iota_{\{V=I(U)\}}}(y_2),\\
&\pi_{i+2}=\textrm{prox}_{\iota_{\mathcal{C}_i}}(y_{i+2})\quad\text{for }i=1,\dots,d,\\
&\bar\pi=\frac{1}{d+2}\Big(\pi_1+\pi_2+\sum_{i=1}^d \pi_{i+2}\Big),\qquad
y_k\leftarrow y_k+\alpha\,(2\bar\pi-x-\pi_k), \\
&x\leftarrow x+\alpha\,(\bar\pi-x).
\end{align}
Here, \(\alpha \in (0, 2), \gamma>0\) are hyperparameters. For pseudocode of the full algorithm, see Appendix~\ref{appendix:algo}.
\begin{remark}[Convergence]\label{rem:convergence}
  The convergence of our iterations to a minimizer follows from the general theory of the Douglas-Rachford algorithm such as \cite{HeYuan2015DRSRate} assuming the feasibility of the problem since the feasibility implies the existence of a minimizer in the discrete problem \eqref{eq:optimization_target} by Theorem \ref{thm:existence_discrete_wfr}. In particular, we have the worst-case $O(1/\sqrt{k})$ decay rate for the fixed point residual by the theory in \cite{HeYuan2015DRSRate}, where $k$ is the number of iterations or $O(1/k)$ error decay in the ergodic sense as follows from the result of \cite{HeYuan2012DRSRate}. Empirically, as showcased by the experiments in Section~\ref{sec:convergence_analysis} of the supplementary material, we find that the decay rate for the distance from the current to the final iterate is typically sublinear, thus suggesting a potentially faster convergence, in practice, than those theoretical estimates. 
\end{remark}

Finally, we need to specify how to compute each proximal operator. We note that the calculations for the \(\textrm{prox}_{\gamma(J+\iota_{\mathcal{CE}})}\) and \(\textrm{prox}_{\iota_{\{V=I(U)\}}}\) are exactly the same as in~\cite{chizat2018interpolating}, so we omit the details here. The new part is the projection onto the (possibly multiple) affine box constraints, which we now describe. For clarity, we shall omit the constraint index $i$ and consider only a single box constraint set: 
\begin{equation}
\mathcal{C}=\{V:\ \ell_{j_0}\le \mathcal{H}(t_{j_0};V) \le u_{j_0}\ \forall\,j_0 \in \llbracket 0,N_0-1 \rrbracket\}
\end{equation}
where, for each time index \(j_0\), $\mathcal{H}(t_{j_0};V)$ denotes, as in the previous section, the linear functional $V \mapsto \langle H^{\rho}_{j_0,\cdot},\rho_{j_0,\cdot}\big\rangle +  \big\langle H^{\omega}_{j_0,\cdot},\omega_{j_0,\cdot}\big\rangle
 + \big\langle H^{\zeta}_{j_0,\cdot},\zeta_{j_0,\cdot}\big\rangle$. Let us also introduce the spatial weights at $t_{j_0}$: $a_{j_0}\coloneqq\big(H^{\rho}(t_{j_0},\cdot),\,H^{\omega}(t_{j_0},\cdot),H^{\zeta}(t_{j_0},\cdot)\big)$. Since different times have disjoint support, \(\{a_{j_0}\}_{j_0}\) form an orthogonal family in the product space $\mathcal{E}_c$.
By~\cite[Theorem~3.3.14 ]{bauschke1996projection}, the projection of \(V\) onto $\mathcal{C}$
is given by the orthogonal-box formula, i.e.
\[
\prox_{\iota_{\mathcal{C}}}(V)= V - \sum_{j_0=0}^{N_0-1}\frac{\lambda_{j_0}(V)}{\|a_{j_0}\|^2}\,a_{j_0},
\]
\[
\lambda_{j_0}(V)=
\begin{cases}
0, & \ell_{j_0}\le \mathcal{H}(t_{j_0};V) \le u_{j_0}\\
\mathcal{H}(t_{j_0};V) - u_{j_0}, & \mathcal{H}(t_{j_0};V)>u_{j_0}\\
\mathcal{H}(t_{j_0};V) - \ell_{j_0}, & \mathcal{H}(t_{j_0};V)<\ell_{j_0},
\end{cases}
\]
where \(\|a_{j_0}\|^2=\sum_{j_1=0}^{N_1-1} \big((H^{\rho}_{j_0,j_1})^2+(H^{\omega}_{j_0,j_1})^2+(H^{\zeta}_{j_0,j_1})^2\big)\,h_1\).

\section{Numerical Experiments}
\label{sec:experiments}
In this section, we demonstrate the capabilities of the proposed framework across various scenarios. These should not be regarded as real-world applications, but rather as simplified examples designed to highlight different aspects of the model. As explained in Remark~\ref{rem:feasibility-experiments} the feasibility (and thus the existence of minimizers) for all examples presented in this section is guaranteed.  
The experimental setup and computation timings are detailed in the supplementary material~\ref{appendix:experiment_details}. Furthermore, in line with good practices of reproducibility, notebooks for all experiments are publicly available at our GitHub repository \url{https://github.com/mao1756/acuot}.

\subsection{Total mass constraints}\label{sec:experiments:mass}
As a first example we consider (equality and inequality) constraints on the total mass, which in the setting of~\eqref{eq:constrained_WFR_extended} corresponds to $H^{\rho} =\pm 1, H^{\omega} = 0$, $H^{\zeta} = 0$ and a ``mass evolution'' function $F(t)$. %

\subsubsection*{Spherical Hellinger-Kantorovich distance}\label{sec:SHK}
We start by discussing the special case of a single equality constraint with $F(t) = 1$, i.e., where $\rho_t$ is constrained to be a probability measure at all times. This situation was first studied by Laschos and Mielke in \cite{laschos2019geometric} and the derived distance is called the \emph{Spherical Hellinger-Kantorovich distance} (SHK). A deep-learning-based algorithm to compute the solution in this situation was recently proposed in \cite{jing2024machine}. Alternatively one can use the results of~\cite[Theorem 2.7]{laschos2019geometric}, who proved that the unconstrained WFR geometry forms a cone over the SHK geometry. This allows one to 
obtain the solution of the constrained problem, by projecting the solution from the \emph{unconstrained} problem onto the space of probability densities. We will use this procedure as a benchmark for our constrained implementation. In Figure \ref{fig:SHK}, we show the solution of the SHK optimization problem between two Gaussian bumps. %
We compare our solution to the projection of the solution to the unconstrained WFR problem using the Laschos-Mielke formula~\cite[Theorem 2.7]{laschos2019geometric}. The two solutions agree both qualitatively and quantitatively; the average~$L^2$ distance between these two minimizers is $\approx 3.0 \times 10^{-3}$. 

\begin{figure}[htbp]
  \centering
 \includegraphics[width=.3\linewidth]{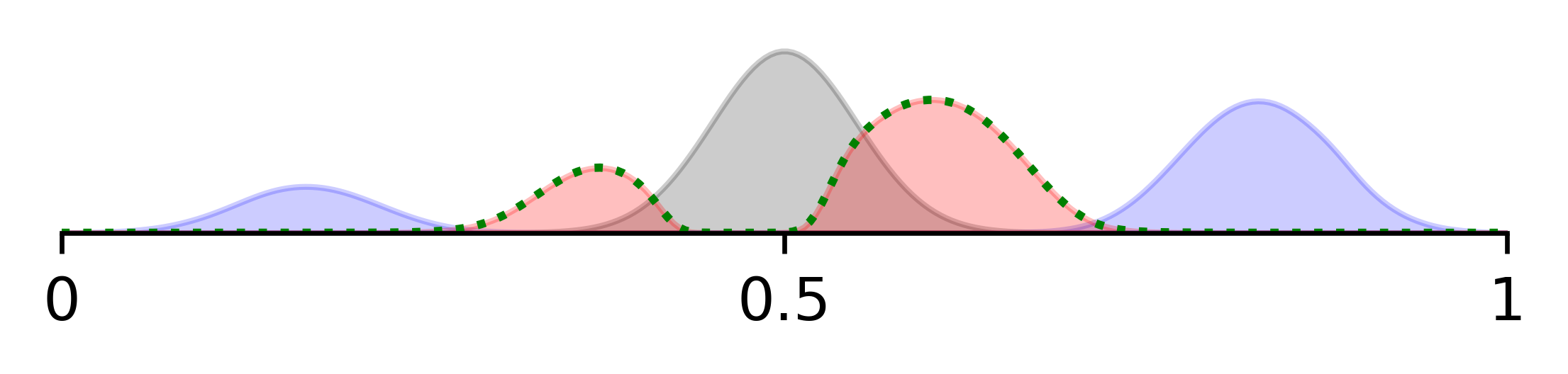}
\includegraphics[width=.3\linewidth]{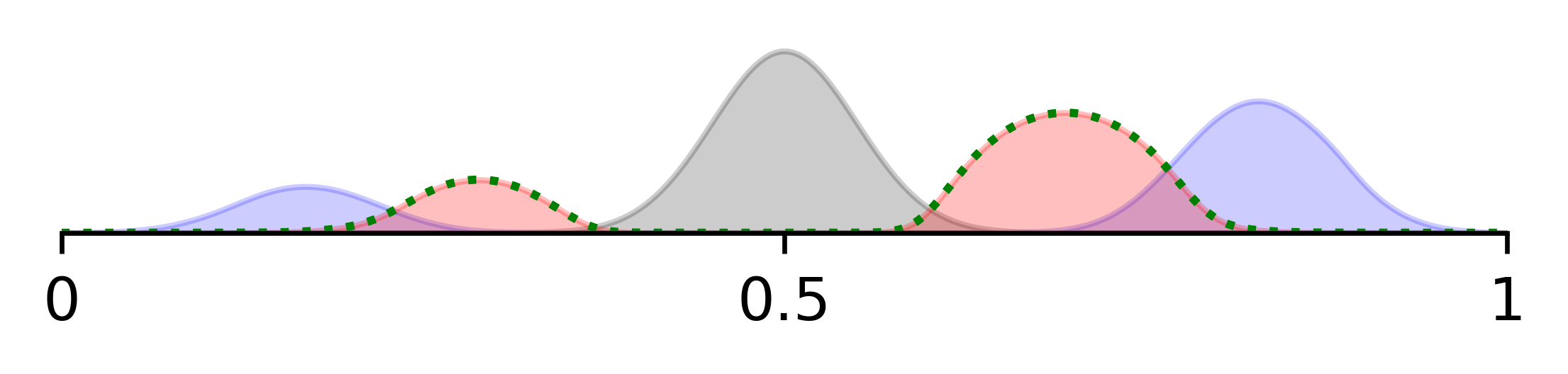}
\includegraphics[width=.3\linewidth]{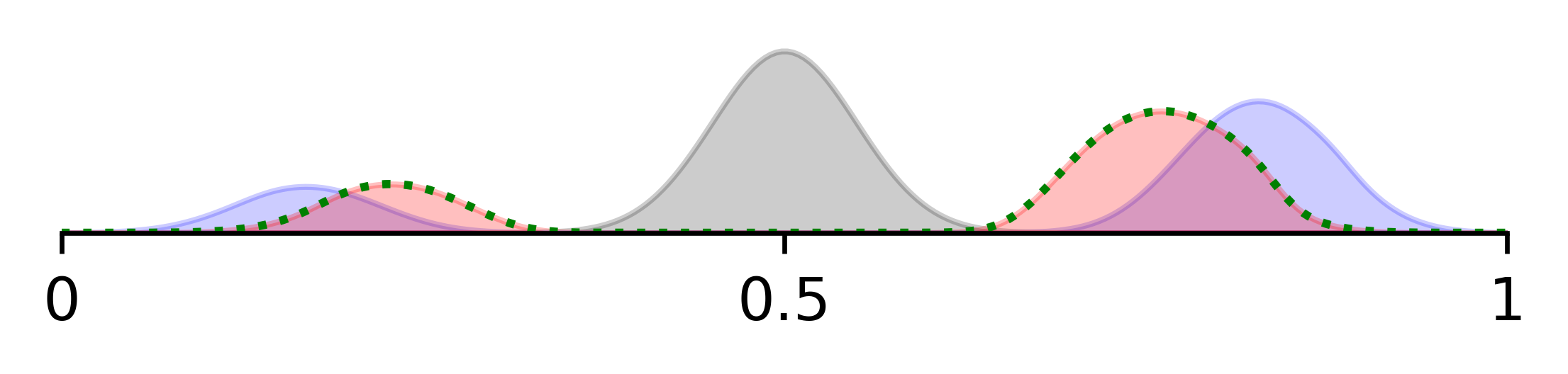}
\caption{Spherical Hellinger-Kantorovich distance: Solution of the constrained WFR problem where the solution is restricted to the space of probability measures. Gray and blue indicate $\rho_0$ and $\rho_1$, respectively; the red curves show the solution at $t=0.25$ (left), $t=0.5$ (middle), and $t=0.75$ (right). The green dotted curve represents the solution of the unconstrained WFR geodesic projected onto the constraint according to the results of \cite{laschos2019geometric}.}
\label{fig:SHK}
\end{figure}

\subsubsection*{Inequality constraints on the total mass}
Next, let us consider a mass inequality constraint of the form 
$\int_{\Omega} d\rho_t(x) \geq c$.
For the specific choices of $c=0.8$ and $c=1$, Figure \ref{fig:totalmass} shows the solutions where the initial and target densities are again chosen as (shifted) Gaussian bumps. Moreover, we display the total mass of the geodesic at each time $t$. In this example, we have set the parameter $\delta$  in the optimization objective to be small enough ($\delta = 0.5/\pi$) so that destroying and creating mass is on average cheaper than transport.  %
\begin{figure}[htbp]  
    \centering
    \begin{subfigure}[t]{0.45\textwidth}
      \includegraphics[width=0.95\linewidth]{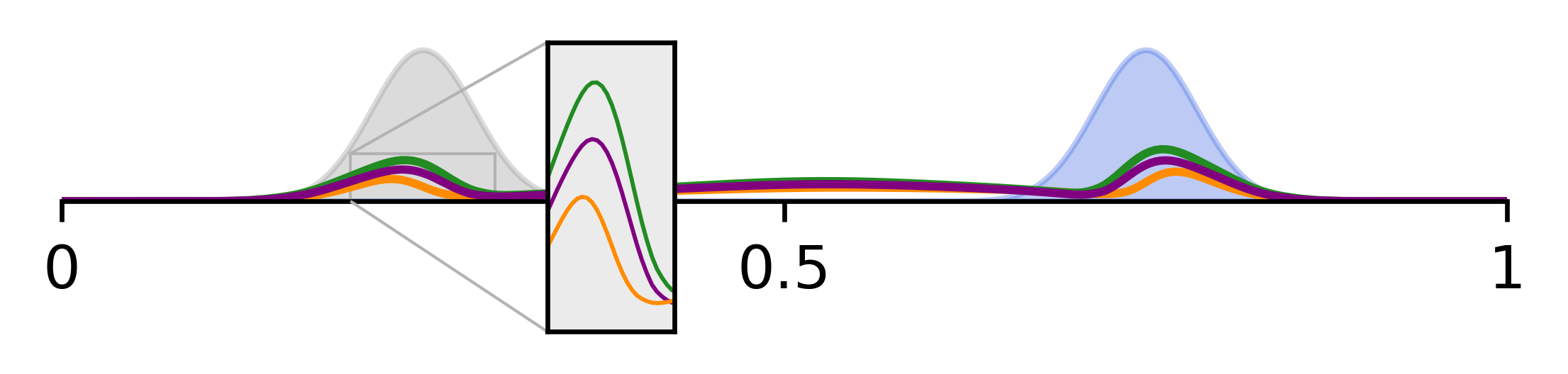}
      \includegraphics[width=0.95\linewidth]{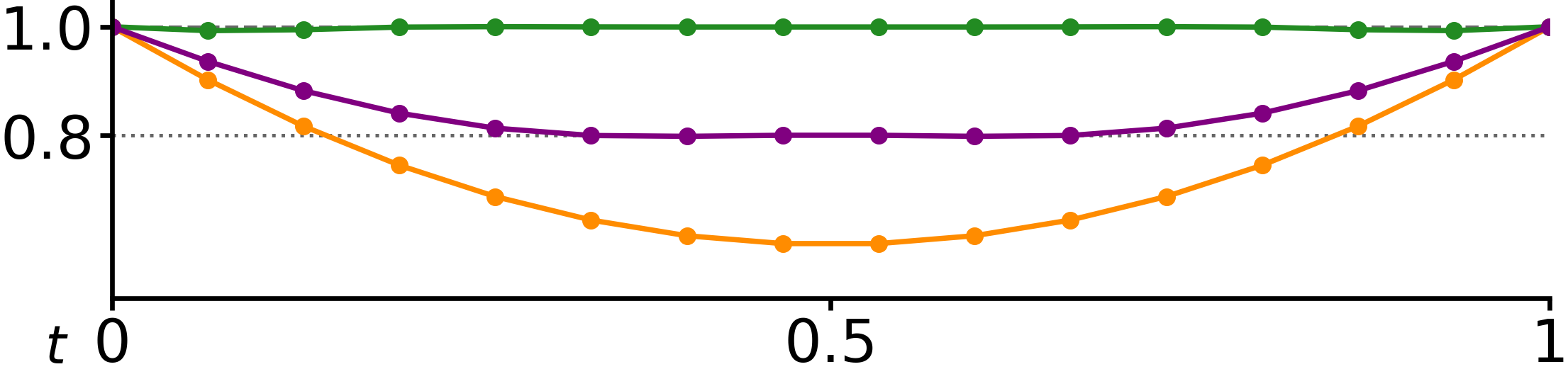}
    \end{subfigure}
    \quad\vline\quad
    \begin{subfigure}[t]{0.45\textwidth}
    \raisebox{-1cm}{
     \includegraphics[width=0.95\linewidth]{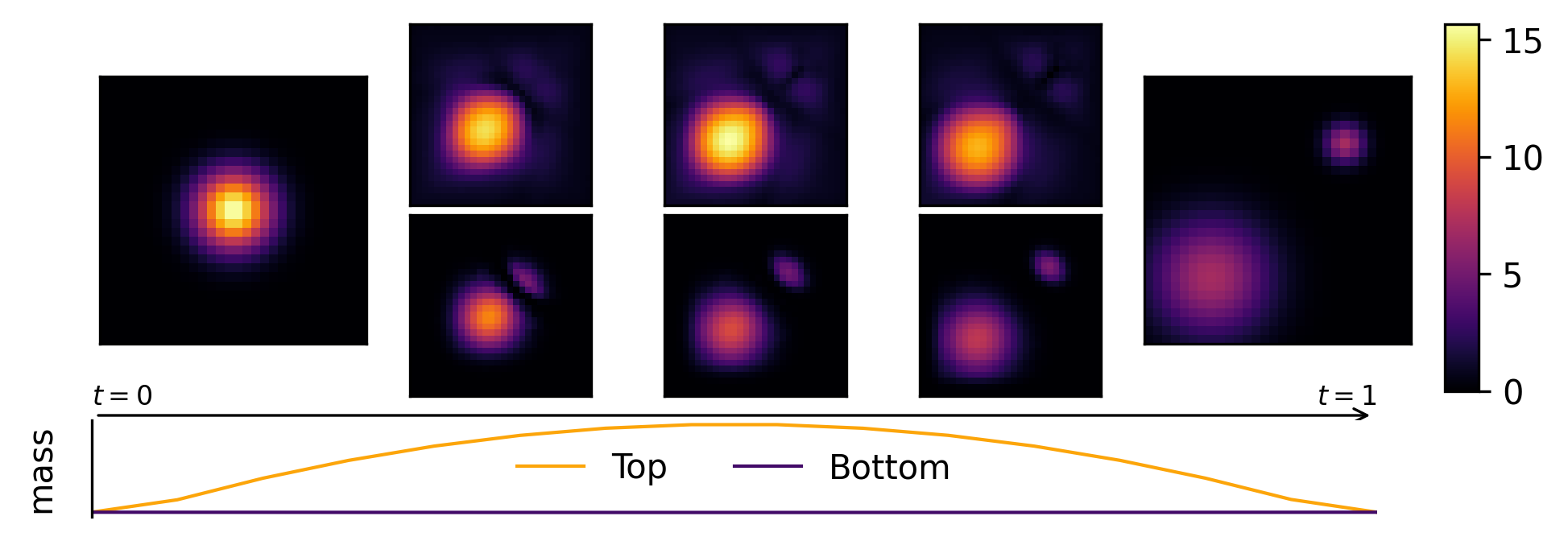}}
    \end{subfigure}
\caption{Left panel: the inequality constraint $\int_\Omega d\rho_t \geq c$, with $c=0.8$ (purple), $c=1.0$ (green), and the unconstrained case (orange). Top figure: gray and blue again correspond to $\rho_0$ and $\rho_1$, while the remaining curves display the solution at $t=0.5$. Bottom figure: Total mass as a function of time for the inequality constraint, with colors matching those in the top figure. Right panel: In 2D, the top panel shows a constrained geodesic and the bottom panel an unconstrained geodesic (evaluated at $t=0$, $0.25$, $0.5$, $0.75$, and $1.0$) along with the evolution of total mass. Here, the total mass constraint is defined by $F(t)=3-8(t-0.5)^2$ for all $t\in[0,1]$.}
\label{fig:totalmass}
\end{figure}

\subsubsection*{2D total mass constraints}
In Figure \ref{fig:totalmass}, we also present an example of the WFR problem with total mass equality constraints in 2D. The domain is $[0,1]^2$, the initial condition is a Gaussian bump at the center, while the terminal distribution is a mixture of two Gaussians. The total mass constraint function is set to $F(t) = 3 - 8 (t - 0.5)^2 $ so that a solution has to increase mass until $t=1/2$ and decrease it back after that. The unconstrained and constrained solutions are shown on the top and bottom row respectively. We observe that the total mass constraint not only influences the total mass of intermediate densities, but also their spatial evolution.

\subsection{Barrier constraints}
\label{subsec:barrier_constraint}
\begin{figure}[htbp]
    \includegraphics[width=\linewidth]{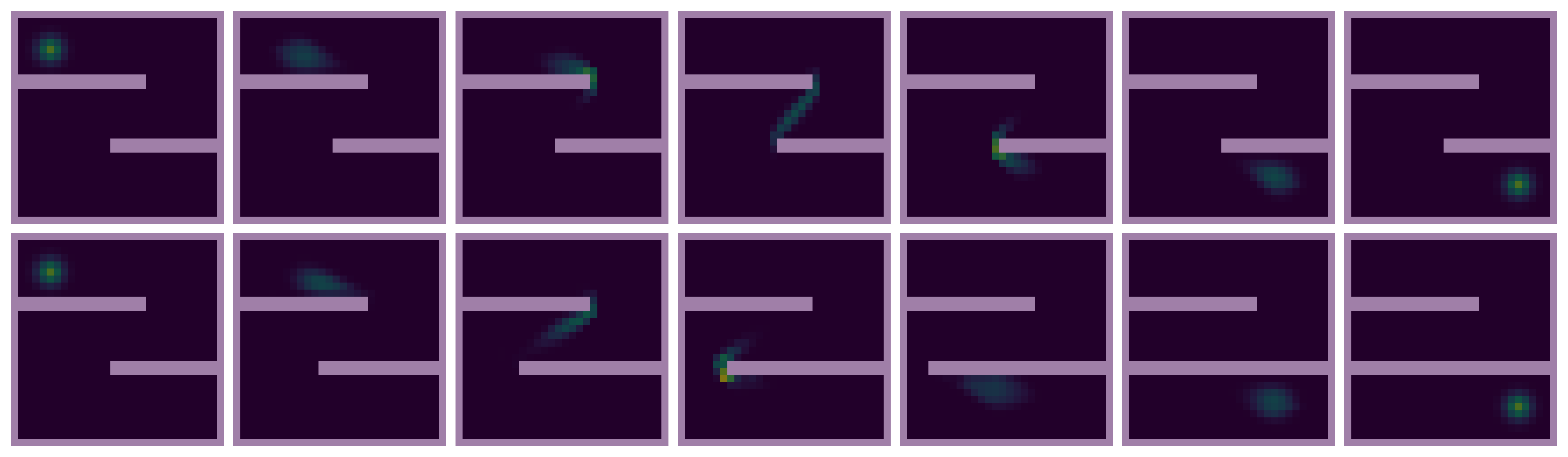}
    \caption{Barrier constraint: transporting a density through a domain with obstacles. The top row corresponds to a static barrier case, while the bottom one shows the results with a dynamic (i.e. time changing) barrier. The solution is shown at $t = 0, 0.1, 0.3, 0.5, 0.7, 0.9$ and $1.0$.}\label{fig:barrier_two_figure}
\end{figure}
In this part, we model the situation of (physical) barriers, i.e., specify (possibly time-dependent) regions of the domain in which no mass can enter or be created. This can be implemented as an affine constraint of the form $\int_{\Omega} H(t,x) \, d\rho_t(x) = 0$,
where $H(t,x)$ is a function that is positive in the prohibited region and zero elsewhere. %

In the following, we present two examples: one with a static barrier and another with a moving barrier. For balanced OT, similar constraints (in the static case) have been previously studied in \cite{papadakis2014optimal}.

Figure~\ref{fig:barrier_two_figure} displays the solution in these two different scenarios where the initial and final distributions are Gaussian bumps located in the upper left and lower right corners of the 2D domain. For both experiments, we choose a relatively high value for the parameter $\delta$ ($\delta=10$) so as to enforce transport over mass teleportation. In the static case (top row of the figure), we see that the distribution slides along the walls to the target density. For the time-dependent barrier (bottom row), the qualitative behavior is similar but with the notable difference that the motion has to accelerate in the earlier times in order to pass the barrier before it closes off the path.

\subsection{Length measures of convex curves}
\label{subsec:convex_curve_constraint}
\begin{figure}[htbp]
  \centering
  \begin{subfigure}[t]{\linewidth}
    \centering
    \includegraphics[width=\linewidth]{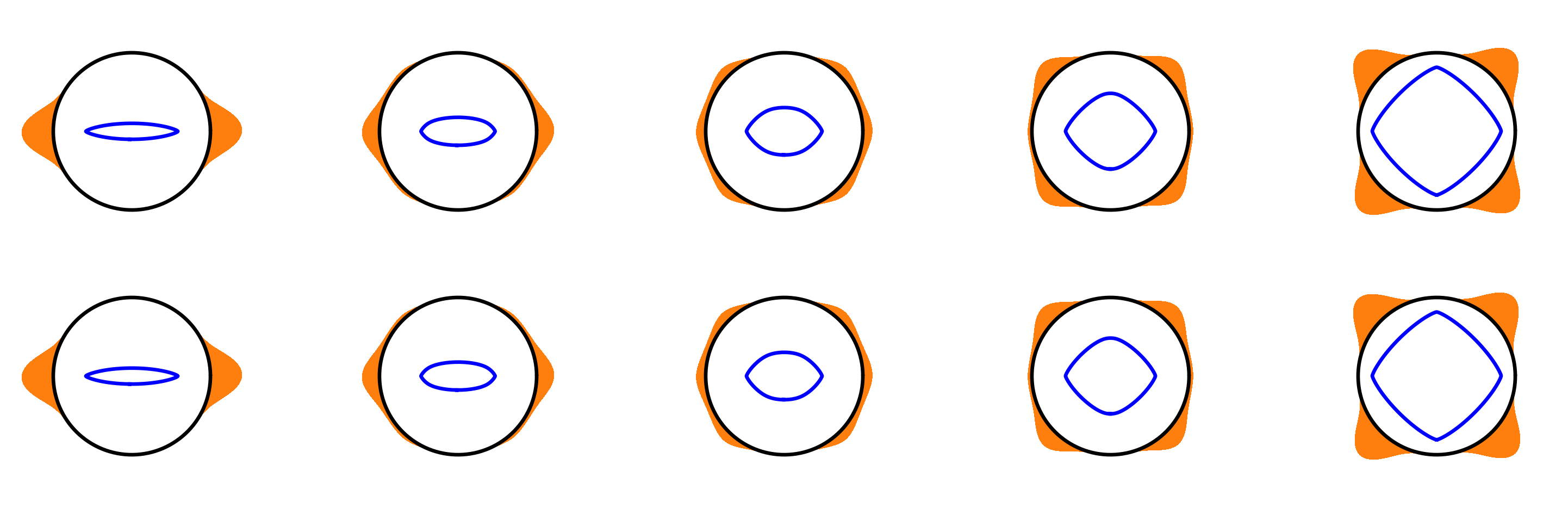}
    \caption{Symmetric case}\label{fig:symmetric_transport}
  \end{subfigure}%

  \vspace{1em}

  \begin{subfigure}[t]{\linewidth}
    \centering
    \includegraphics[width=\linewidth]{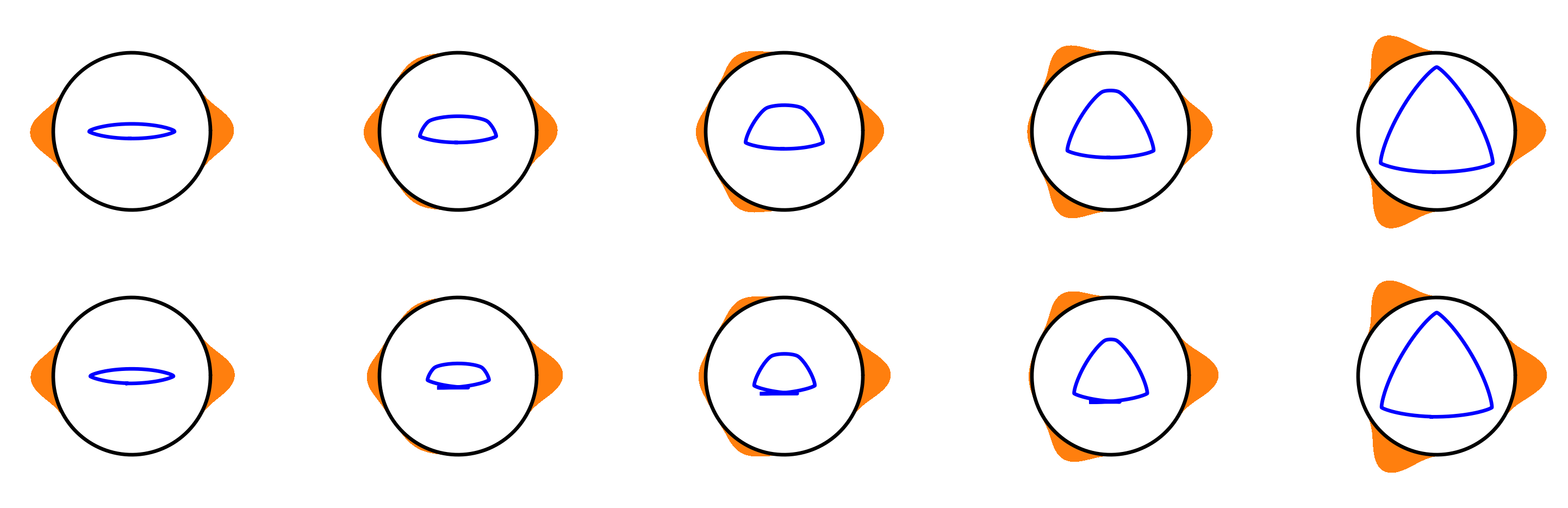}
    \caption{Nonsymmetric case}\label{fig:nonsym_transport}
  \end{subfigure}
  \caption{Effect of a convex-curve constraint on the geodesic evolution. From left to right we show 
$t=0, 0.25, 0.5, 0.75,$ and $ 1$. Each panel displays the density on $\mathbb{S}^1$
 (orange) together with the corresponding convex curve (blue). The top row enforces a convex-curve constraint; the bottom row shows the unconstrained geodesic.}
\end{figure}
{In this example we aim to use our framework to define a notion of optimal deformation between convex shapes in the plane. To that end, we assign to each convex curve its corresponding length measure $\mu$, defined as the measure on the unit circle $\mathbb{S}^{1}$ given by the pushforward of the arclength measure by the Gauss map;}
in other words, for any measurable set $A \subset \mathbb{S}^1$, $\mu(A)$ represents the total length of the portion of the convex curve for which the unit tangent vectors lie within $A$. Due to a classical theorem of Minkowski \cite[Section 8.2.1]{Schneider2014}, $\mu$ must satisfy:
\begin{equation}
  \label{eq:convex_set_constraint}
  \int_{\mathbb{S}^{1}} xd\mu(x) = 0.
\end{equation}
Conversely, if a measure $\mu$ on $\mathbb{S}^{1}$ satisfies this constraint, then there exists a convex curve, unique up to translations, whose length measure is $\mu$~\cite[Theorem 8.2.2]{Schneider2014}. Thus, by considering the WFR energy subject to the constraint~\eqref{eq:convex_set_constraint} (i.e. one equality constraint with $H(x) = x$ and $F(t) = 0$), one obtains a corresponding metric on the space of convex curves in $\mathbb R^2$. For the space of convex, unit length curves, this was considered in~\cite{charon_length_2021} within the balanced optimal transport setting. Here, we extend this idea to unbalanced optimal transport, thus extending the metric to the space of all convex curves, not only unit length ones. 

In Figure \ref{fig:symmetric_transport} we show an example, where the
source and the target curve (measure, resp.) are both symmetric with respect to the antipodal map $R(x) = -x$, i.e., the source and target measures $\rho_0, \rho_1$ satisfy $R_* \rho_0 = \rho_0, R_* \rho_1 = \rho_1$ where $R_*$ is the pushforward map by $R$. In this case, we observe that the unconstrained algorithm leads to essentially the same solution as the constrained algorithm, with the $L^2$ distance between the constrained and unconstrained solution being $8.7\times 10^{-8}$. Note that this behavior aligns with theoretical expectations: in this setting, one can easily verify the existence of symmetric minimizers for the unconstrained problem, which therefore automatically satisfy the constraint~\eqref{eq:convex_set_constraint}. For the reader's convenience, we provide the short justification of this statement in the supplementary material.

In Figure \ref{fig:nonsym_transport}, we illustrate that for non-symmetric measures, there is, however, a significant difference between the constrained and unconstrained problem. We specifically consider the same initial measure as in our last example, but change the terminal measure to a mixture of three Gaussians on $\mathbb{S}^1$, breaking the symmetry w.r.t. the mapping $R$. We observe that the unconstrained geodesic violates the convex curve constraint, and consequently, the reconstructed interpolating curves are not closed. On the other hand, the constrained geodesic satisfies the convex curve constraint at all times, and thus we obtain a valid interpolation in the space of convex, closed curves. Despite the density plots appearing visually similar, the $L^2$ distance between the constrained and the unconstrained solution is here $1.3 \times 10^{-1}$. We leave it to future work to study more in depth the properties and potential applications of the resulting metric on convex sets, as well as the extension of our numerical framework for convex surfaces and area measures on $\mathbb{S}^2$.

\subsection{Constraints on control variables}
The constraints considered so far only applied to the density path $\rho$. However, our framework can also handle constraints on the control variables $\omega$ and $\zeta$. For instance, one may want to constrain the angle of the momentum $\omega$, or limit the amount of mass created at each time by constraining $\zeta$. In this section, we present two different examples of such constraints.

\subsubsection{Momentum flow constraint}
\label{ssec:experiments:flow}
\begin{figure}
\centering
\includegraphics[width=0.8\textwidth]{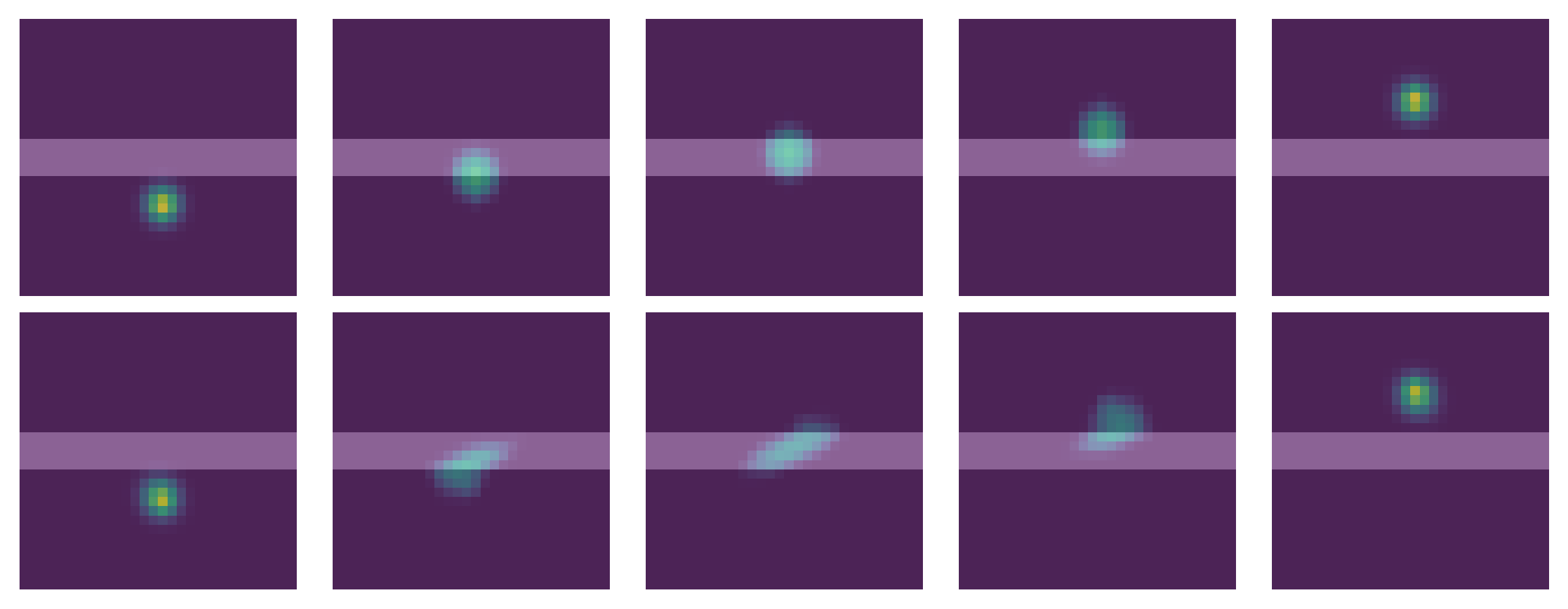}
\caption{Illustration of a flow constraint: the brighter region represents where the momentum is constrained to align with the given vector field. The bottom row shows the constrained geodesic, while the top row shows the unconstrained geodesic. Each panel displays the density at $t=0, 0.25, 0.5, 0.75,$ and $1$.}
\label{fig:river_constraint}
\end{figure}

In this example, we consider a constraint of the following form:
\begin{equation}
  \int_{\Omega} H^{\omega}(t, x) \cdot d\omega_t(x)  \geq 0
\end{equation}
which can be interpreted as enforcing the momentum $\omega$ to loosely align with the prescribed vector field $H^{\omega}(t,x)$, effectively modeling a form of extrinsic current flow such as a river. In Figure~\ref{fig:river_constraint}, we compare transport between localized 2D Gaussian distributions in the constrained and unconstrained settings. For the constrained solution, the vector field $H^{\omega}(t,x)$ points bottom-right in the purple shaded area and is zero elsewhere. As a result, $\omega(t,x)$ must stay within the halfspace of vectors positively correlated with  $H^{\omega}(t,x)$. Notably, we see that the solution adjusts the momentum direction prior to entering the region as a way to compensate for this effect.

\subsubsection{Budget constraint on mass creation/destruction}
\label{ssec:experiments:budget}
\begin{figure}[htbp]
  \centering
  \begin{subfigure}[t]{.48\textwidth}
    \centering
    \includegraphics[width=\textwidth]{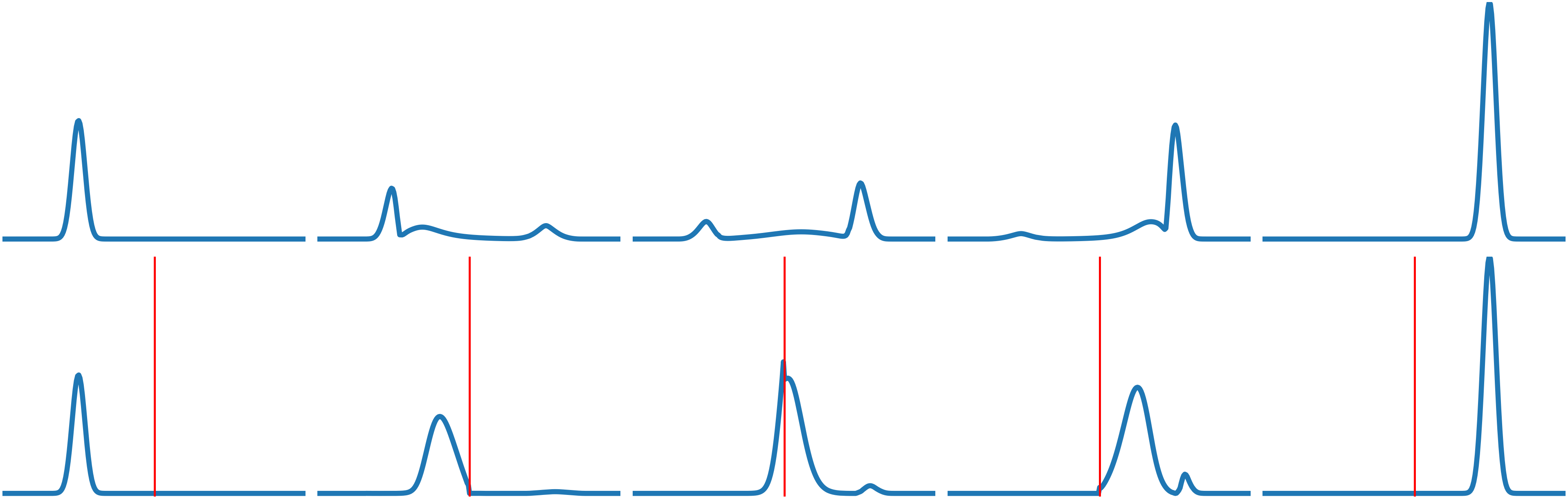}
    \caption{}
    \label{fig:budget_constraint}
  \end{subfigure}
  \begin{subfigure}[t]{.48\textwidth}
    \centering
    \includegraphics[width=\textwidth]{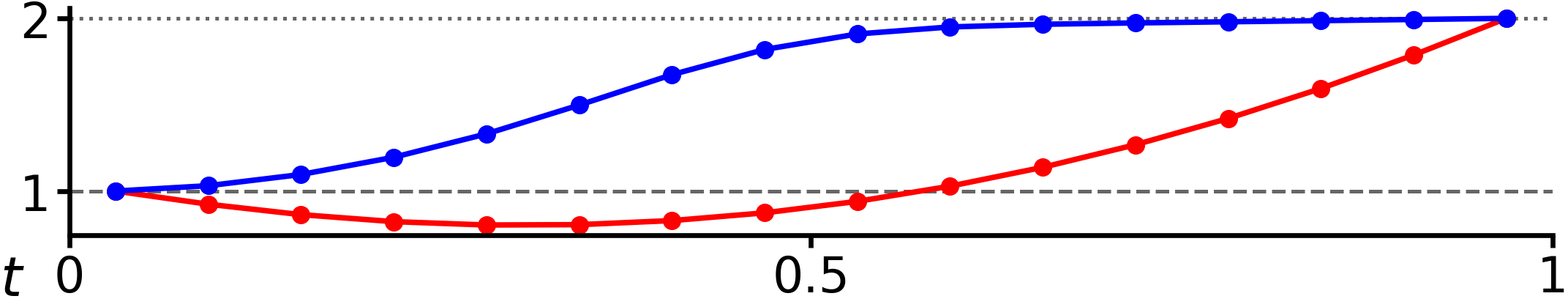}
    \caption{}\label{fig:budget-constraint:panelb}
  \end{subfigure}
  \caption{Panel \subref{fig:budget_constraint}: Budget constraint on mass creation/destruction. The portion to the right of the red line is where the mass creation/destruction is constrained to be small. Top and bottom rows show the unconstrained and constrained solutions, respectively. Each row displays the density at $t=0, 0.25, 0.5, 0.75,$ and $1$. Panel~\subref{fig:budget-constraint:panelb}: The total mass at each time point for the unconstrained (red) and constrained (blue) optimal paths.}
\end{figure}

In this example, we consider a constraint on the source term $\zeta$, of the form:
\begin{equation}
  0 \leq \int_{\Omega} H^{\zeta}(t, x) d\zeta_t(x) \leq  C
\end{equation}
for some constant $C>0$. This constraint effectively limits the amount of mass that can be created or destroyed at each time, modeling scenarios where resources are limited. In Figure~\ref{fig:budget_constraint}, we compare transport between 1D Gaussian distributions in both constrained and unconstrained settings. For the constrained case, the function $H^{\zeta}(t,x)$ is set to $1$ in the right half of the domain and $0$ elsewhere, effectively limiting mass creation/destruction in that region. In this experiment, we set $C = 0.1$. As a result, one can observe that what was mostly a mass creation in the unconstrained model becomes mostly transport in the constrained case. In particular, as evidenced by the total mass plot, the constrained optimal path gradually increases the mass until reaching the constraint region, and maintains a nearly constant mass after that. This example illustrates how constraints on mass creation/destruction can significantly alter the transport dynamics.

\subsection{Mixed constraints: population transport}\label{ssec:experiments:population}
\begin{figure}[htbp]
  \centering
  \begin{subfigure}[t]{\linewidth}
    \includegraphics[width=\linewidth]{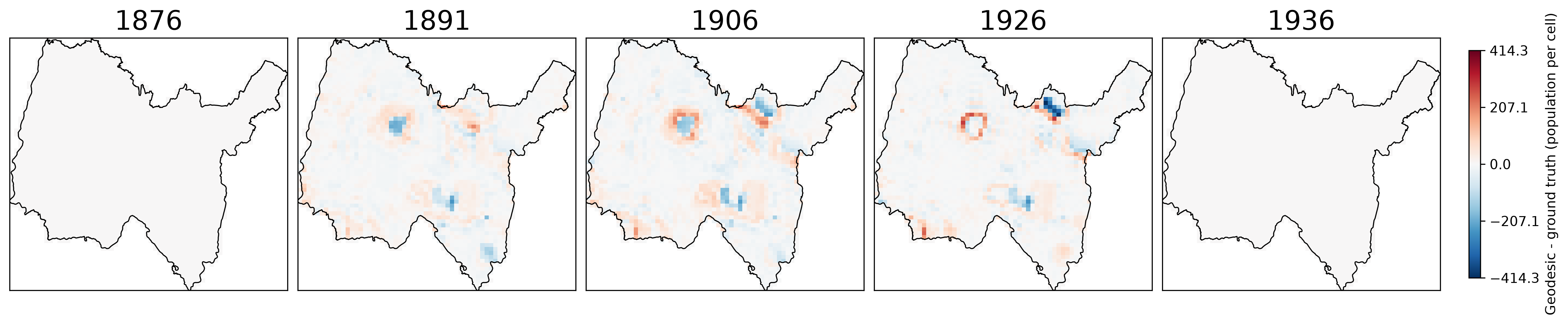}
    \caption{}
    \label{fig:ain_pop}
  \end{subfigure}
  \hspace{0.05\linewidth}
  \begin{subfigure}[t]{\linewidth}
    \includegraphics[width=\linewidth]{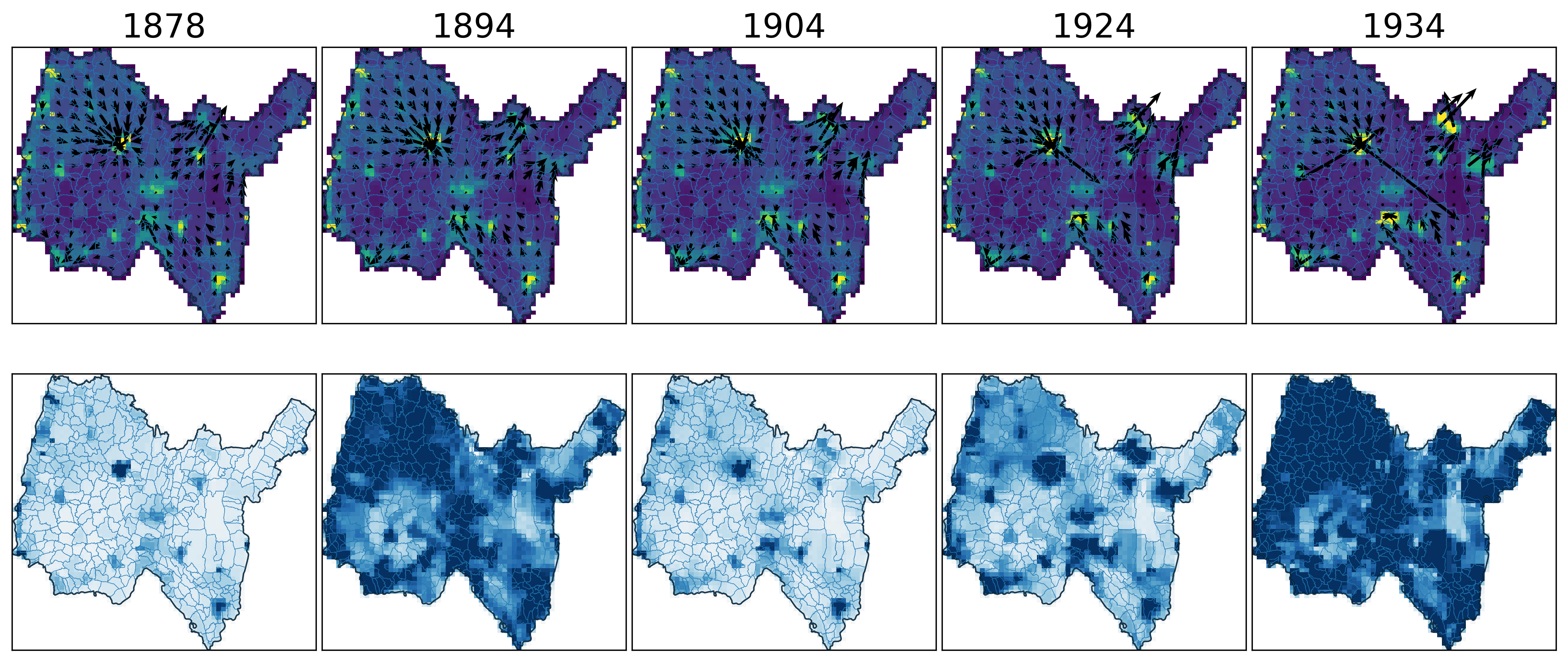}
    \caption{}
    \label{fig:ain_momentum_source}
  \end{subfigure}
  \caption{\subref{fig:ain_pop} %
  {The discrepancy map for the population distribution of Ain in France, at different years. The figure shows the difference between the population interpolated by our algorithm between the first (1876) and final year (1936) and the ground-truth data. Since the terminal densities at years 1876 and 1936 are enforced exactly, they show no discrepancy. The largest discrepancy occurs around the city of Oyonnax (top right), where the real data shows a rapid increase in population. } \subref{fig:ain_momentum_source} Estimated momentum and density values (top) as well as the source term (bottom) across time. The momentum is represented by arrows, with length proportional to the momentum magnitude. The source values have been clipped at the 80th percentile to mitigate the effect of outliers and are represented by color, with darker blue indicating stronger population decrease. Those show population migration towards the cities of Bourg-en-Bresse (center) and Oyonnax (top right). On the other hand, the source term illustrates the important population decrease in rural areas. }
\end{figure}
In this section, we demonstrate how to impose multiple constraints simultaneously on a real data example. Specifically, we are interested in interpolating between population density maps of the particular county of Ain in France. While we are not claiming that optimal transport is necessarily the proper model for this type of application, this example is mostly here to emphasize how our numerical framework can handle multiple constraints on non-synthetic distributions. 

Specifically, we consider the population density maps of Ain between the years 1876 and 1936 extracted from the INSEE records \cite{insee_pop_communes_1876_2020} (with censuses made at regular five-year intervals), along with the total county population over that period. Our goal is to interpolate between the distributions of 1876 and 1936 and evaluate how the path differs from the ground-truth density evolution. To that end, we impose two types of constraints: a barrier constraint to ensure that the mass remains within the region of interest (Ain) and a total mass constraint enforcing that the total population evolves according to the known records. 

Figures \ref{fig:ain_pop} and \ref{fig:ain_momentum_source} present the population distributions for both the actual data and the optimal interpolation path, as well as the momentum and the source term at various time points. While the general trend of population change is well captured (this particular county witnessed overall population decrease during that period), this approach does not entirely reflect some rapid changes in certain urban areas, such as Oyonnax (top right), between 1906 and 1936. These discrepancies may be attributed to the long time period considered here, which could obscure non-optimal local migration patterns, as well as to the restrictive barrier constraint that excludes in- or out-flows from external regions, among other things.

{
\subsection{Decomposing colony dispersion into transport and proliferation for the TOIAM cell tracking data}
\begin{figure}[ht]  
\centering\includegraphics[width=\textwidth]{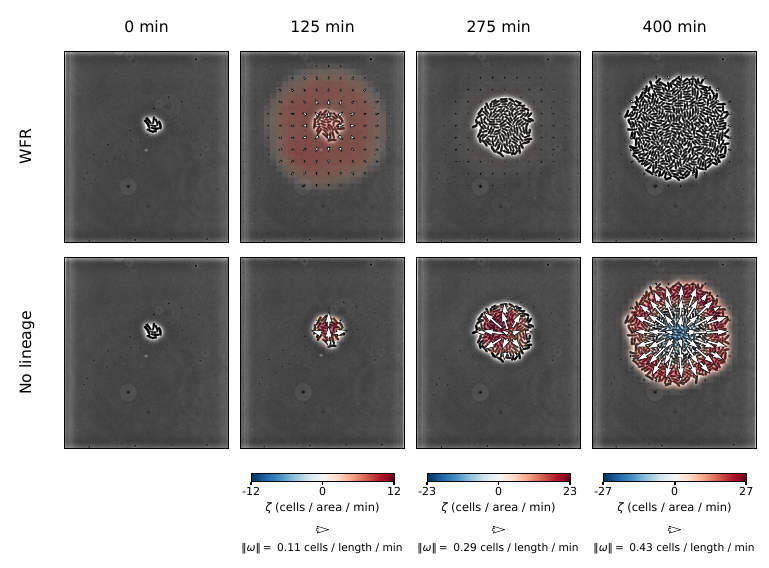}
    \caption{{The source term (shown by red and blue colors) and the momentum term (shown by the white arrows with black outline) inferred from the TOIAM dataset and displayed on top of the dataset cell images at four time points. The first row (WFR) shows the unconstrained result, while the second row (no lineage) corresponds to the model with moment equality constraints \eqref{bio:constraints}. Red color is associated to mass creation while the arrows show the inferred direction of proliferation at each location. One can see that, in the unconstrained WFR case, the source term $\zeta$ (primarily at the early times) largely dominates the evolution compared to $\omega$. The constrained version, on the other hand, leads to more balanced and natural trajectories, with the source and momentum estimates staying concentrated inside the cell colony itself.}}
    \label{fig:toiam_sequence}
\end{figure}
The Tracking One-in-a-Million (TOIAM) dataset \cite{seiffarth2025tracking} contains five time-lapse sequences of growing \textit{Corynebacterium glutamicum} in Cell Tracking Challenge (CTC) format. Each sequence consists of cell images captured by a microscope in one-minute intervals (an example is shown in Figure \ref{fig:toiam_sequence}), the segmentation mask for each cell, and lineage information, including the cell ID numbers and parent-daughter relationships. In what follows, our goal will be to assess whether our UOT framework can be used to quantify how much of the observed change in colony dispersion is attributable to spatial transport vs how much is the result new cells creation via division. We address this question under the plausible assumption that experimentalists may have access to only partial information about the experiment, e.g., in the form of global statistics of the cell growth (see details below). We shall then evaluate the extent to which incorporating this partial information improves our model's prediction.

To convert the cell images into a format compatible with our framework, we first replace, for each time point $t$, each segmented cell with a fixed-width unit mass Gaussian centered at the cell's centroid, and define the density $\rho_t$ to be the superposition of these Gaussians.
Note that this modeling approach assigns exactly one unit of mass to each cell, thereby ignoring differences in cell size and focusing solely on cell counts. In addition, we will assume that one has access to only limited information about the ground-truth density $\bar{\rho}$ at intermediate time steps, namely, the first three moments of the distributions, which leads to the affine equality constraints:
\begin{align}\label{bio:constraints}
    \int_{\Omega}d\rho_t = \int_{\Omega}d\bar{\rho}_t, \quad 
    \int_{\Omega}xd\rho_t = \int_{\Omega} xd\bar{\rho}_t, \quad
    \int_{\Omega}|x|^2d\rho_t = \int_{\Omega}|x|^2d\bar{\rho}_t\;. 
\end{align}
Fig.~\ref{fig:toiam_sequence} show the estimated transport and source terms obtained for both the unconstrained and constrained WFR problem. They illustrate how simply imposing these three global constraints already leads to more realistic patterns for the estimated transport and source proliferation.  

Next, we examine whether those results can be further improved by adding some extra constraints based on partial knowledge of the lineage information. To that end, let us introduce the following quantities:
\begin{equation}
    N(t) = \int_{\Omega}d\rho_t, \quad \mu(t) = \frac{1}{N(t)}\int_{\Omega}xd\rho_t, \quad O(t) = \frac{1}{N(t)}\int_{\Omega}|x-\mu(t)|^2d\rho_t.
\end{equation}
Here, $\rho_t$ is the density at time $t$, $\Omega$ is the full domain of the problem, $N(t)$ is the total mass of $\rho_t$ (a proxy for the \emph{number} of cells at time $t$) and $O(t)$ is the dispersion measure (variance) that we will use in the remainder of the section. Note that, due to \eqref{bio:constraints}, $N(t)$, $\mu(t)$ and $O(t)$ all match the same quantities for the ground-truth density $\bar{\rho}$ and therefore are assumed to be known here. We then decompose the time variation $dO/dt$ into its corresponding transport and proliferation contributions, i.e.
\begin{equation}\label{eq:TandS}
    \frac{dO}{dt} = T(t) + S(t),
\end{equation}
where $T(t)$ only involves the momentum $\omega$ (corresponding to transport) and $S(t)$ only the source term $\zeta$ (corresponding to proliferation). Indeed, using the continuity equation and the constraints~\eqref{bio:constraints}, one can eliminate all terms involving $\rho$ and obtain, after calculation, the decomposition \eqref{eq:TandS} with:
\begin{equation}
    T(t) = \frac{1}{N(t)}\int_{\Omega}2(x-\mu(t))\cdot d\omega_t, \quad S(t) = \frac{1}{N(t)}\int_{\Omega}(|x-\mu(t)|^2-O(t))d\zeta_t.
\end{equation}
Intuitively, $T(t)$ is positive when the momentum $\omega$ is pointing away from the centroid $\mu(t)$ of the cell colony, and similarly $S(t)$ is positive if when mass is being created away from the currently observed dispersion $O(t)$.

Now, we consider the case where lineage information is (partially) available which we shall use to enforce confidence interval bounds on the term $S(t)$ above. For a time interval of length $dt$, suppose that there were $D$ division events occurring at times $\tau_i$ for $i=1,\cdots, D$. At $\tau_i$, we make the assumption that each daughter cell contributes 1/2 mass to $\zeta$ in a form of a Gaussian, in other words:
$
    \zeta(dt\times dx) = \sum_{i=1}^{D}\delta_{\tau_i}(dt)\zeta_i(x)dx$,
    with $\zeta_i(x) = \frac{1}{2}G_\sigma(x-c_{i,1}) +\frac{1}{2}G_\sigma(x-c_{i,2}),$ where $\delta_{\tau}$ is the Dirac measure at $\tau$, $G_\sigma$ is the Gaussian function with standard deviation $\sigma$, and $c_{i,1},c_{i,2}$ represent the centroids of the daughters from the $i$th division event. Thus, we have informally:
\begin{equation}
\label{eq:ref_S}
    S(t)dt = \frac{1}{N(t)}\sum_{i=1}^{D}\int_{\Omega}(|x-\mu(t)|^2-O(t))\zeta_i(x)dx:=\sum_{i=1}^{D}S_i(t).
\end{equation}
To model the partial lineage information, we assume that only the first $p$ of the $D$ division events are known, i.e., that we observe $S_i(t)$, $i=1,\cdots, p$. Estimating $S(t)$ from known $S_i$ is the classical problem of estimating a finite population mean from the sample mean. Thus we can calculate the following confidence interval for $S(t)$:
\begin{equation}
    \frac{D}{dt} \hat{S}-1.96\frac{D}{dt}\sqrt{\left(1-\frac{p}{D} \right)\frac{s^2}{p}}\leq S(t) \leq \frac{D}{dt} \hat{S}+1.96\frac{D}{dt}\sqrt{\left(1-\frac{p}{D} \right)\frac{s^2}{p}}
\end{equation}
Here, $\hat{S}$ is the sample mean of $S_i$ for known division events, and $s^2$ is the sample variance for the same group. Based on the expression of $S(t)$ above, these two inequalities both induce affine constraints on the source term $\zeta$ that are within our proposed model.

Under these constraints, we compute the Wasserstein-Fisher-Rao geodesic between the true density at the initial $t=0$ and final time $t=400 \text{ mins}$, and use the calculated $\omega$ and $\zeta$ to estimate $T(t)$ and $S(t)$. Figure \ref{fig:toiam-result} displays the graphs of the deviations of the estimated $T(t)$ and $S(t)$ to their reference values (i.e. those obtained from \eqref{eq:ref_S}) for each of the following cases: unconstrained WFR and constrained WFR with the moment constraints \eqref{bio:constraints} plus the previous lineage constraints with, respectively, 0, 30, 50, 80 and 100\% of the lineage information. For the parameter $\delta$, the value $\delta=0.05$ was selected consistently throughout the experiment. We remark that the full-lineage geodesic does not coincide exactly with the reference because the computed solution does not fully satisfy the imposed constraint. This residual is numerical rather than mechanical and should decrease as convergence improves. Notably, the unconstrained geodesic predicts \(T(t)\) and \(S(t)\) with errors of order \(10^{-3}\), while imposing the moment constraints alone reduces these errors to order \(10^{-4}\). Revealing only \(30\%\) of the lineage information yields a further substantial reduction, and increasing the percentage improves the residual, demonstrating the effectiveness of the inequality constraints on the source control.
Spatial visualizations of the inferred momentum $\omega$ and source term $\zeta$ for these settings are provided in Appendix~\ref{appendix:toiam_controls}, Figures~\ref{fig:toiam_controls_200_250}--\ref{fig:toiam_controls_550_600}.}

\begin{figure}[ht]
\centering\includegraphics{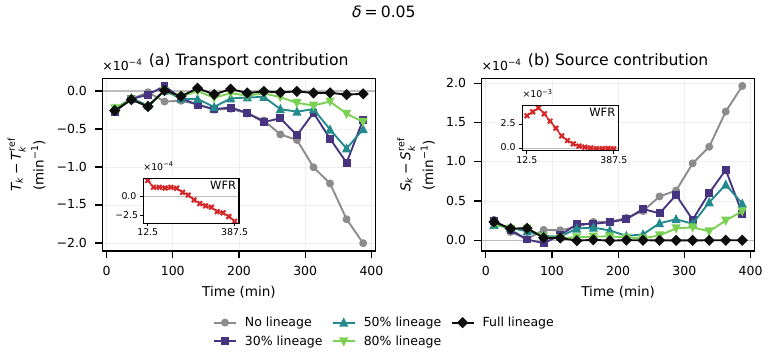}
    \caption{Lineage-informed decomposition of colony dispersion. Differences between the inferred (a) transport $T(t)$ and (b) source contributions $S(t)$ and their withheld-lineage references are shown for ACUOT using no lineage, 30\%, 50\%, 80\%, and full lineage information; zero indicates perfect agreement. The source reference $S^{\mathrm{ref}}$ is computed from all observed division events, and the transport reference is defined by $T^{\mathrm{ref}}=\Delta O/\Delta t-S^{\mathrm{ref}}$. Results for unconstrained WFR is shown separately in the insets as its deviations require a different vertical scale. }
    \label{fig:toiam-result}
\end{figure}

\section{Conclusions, Future Work and Limitations}\label{sec:conclusion}
In this article we introduced a variant of the dynamic unbalanced optimal transport model which allows one to enforce a set of time-dependent integral equality and inequality constraints along the density, vector field and/or source term paths. This formulation has the important advantage of yielding convex variational problems, with provable conditions of well-posedness in both the continuous and discrete settings. These conditions are still tied to the issue of feasibility (i.e., the existence of finite energy paths satisfying all constraints), for which a more general and systematic study, in the context of our model, remains a subject of possible future work.

Based on a discretization over staggered time-space grids, we then proposed an adaptation of the parallel proximal point scheme to tackle the corresponding class of constrained non-smooth optimization problems. We showed empirically, on various examples and constraints, that despite the dimensionality of those problems, the approach is able to estimate an approximate solution in a robust and consistent way. These examples are also meant to illustrate the variety of dynamic constraints that are subsumed under our model, and their potential relevance in different applications. 

Yet we should also mention a few avenues for further improvement and extension of the framework presented in this paper. A first possible limitation is the specific constraint type that is considered here. It does not include, for instance, integral convex inequality constraints (beyond affine ones) such as bounds on the entropy or variance of the density $\rho_t$, or certain forms of congestion constraints in crowd navigation applications \cite{kerrache2013optimal,ruthotto2020machine}. Although these would technically still lead to convex formulations, the generalization of Theorem \ref{thm:main_theorem} or \ref{thm:existence_discrete_wfr} to this new setting is, to our knowledge, not straightforward and thus could be an interesting subject of future investigation. On the numerical side, we expect the PPXA algorithm of Section \ref{ssec:ppxa_algo} to be still relevant, provided that expressions of the proximal operators for the new constraint functions are available. Another path worth exploring is the extension of the constrained transport model of this paper to the stochastic setting, in particular that of Schr\"{o}dinger bridge matching and its recent unbalanced version known as generalized Schr\"{o}dinger bridge matching \cite{liu2024generalized,chen2025optimal}, which share many connections to diffusive flows and generative modeling in machine learning.  

\section*{Acknowledgements}
M. Nishino was partially supported by NSF grant CISE--2426549. M. Bauer and T. Needham were partially supported by NSF grants DMS--2324962 and CIF--2526630. N. Charon was partially supported by NSF grants DMS--2438562 and CIF--2526631. 

{ChatGPT (GPT-5.1--6) assisted in the implementation of the numerical framework developed in this article. In particular, it was used for the conversion of the Julia implementation from~\cite{chizat2018interpolating} to Python. Moreover, the same system was used to explore the feasibility of the problem; this has led to initial formulations  of Theorems \ref{thm:feasibility_density_constr} and \ref{thm:feasibility_controls_constr}, which were subsequently refined and completed by the authors. In addition, we used ChatGPT to polish the authors’ written text for spelling and grammar. The authors have verified all mathematical statements and arguments in the manuscript and take full responsibility for its contents.}

\bibliographystyle{alpha}  
\bibliography{numerics_biblio} 

\pagebreak
\renewcommand{\appendixname}{Supplementary Material}
\setcounter{page}{1}
\appendix\label{supplementary_material}
\begin{center}
\large Supplementary Material for the article \emph{``A Benamou-Brenier proximal splitting method for constrained UOT''} \\M Nishino, M. Bauer, T. Needham and N. Charon
\end{center}
\thispagestyle{empty}
\section{Experimental Details}
\label{appendix:experiment_details}
\subsection{Environment}
\begin{description}
  \addtolength{\itemindent}{0.25in}
  \item[OS:] Windows 11 Pro (Build 26100).
  \item[Python:] \texttt{〈3.12.7〉}.
  \item[Libraries:] \texttt{numpy}~\texttt{〈2.1.3〉}; \texttt{scipy}~\texttt{〈1.14.1〉}; \texttt{POT}~\texttt{〈0.9.5〉}.
\end{description}

\subsection{Hardwares}
{\bf System A} Used for all experiments except TOIAM.
\begin{description}
  \addtolength{\itemindent}{0.25in}
  \item[System:] Dell G15 5530 laptop (x64).
  \item[CPU:] Intel Core i7-13650HX (20 CPUs), $\sim$\,2.6\,GHz.
  \item[Memory:] 16\,GB RAM.
\end{description}

\smallskip

{\bf System B} Used for the TOIAM experiments.
\begin{description}
  \addtolength{\itemindent}{0.25in}
  \item[System:] ASUS TUF Gaming F16 FX608JMR.
  \item[CPU:] Intel Core i7-14650HX (24 CPUs), $\sim$\,2.2\,GHz.
  \item[Memory:] 32\,GB RAM.
\end{description}

\subsection{Hyperparameters}
The global hyperparameters used in all experiments are as follows:
\begin{description}
  \addtolength{\itemindent}{0.25in}
  \item[Relaxation parameter:] $\alpha = 1.8$.
  \item[Proximal parameter:] $\gamma = \max(\max(\rho_0) ,\max(\rho_1))/2$ where $\rho_0$ and $\rho_1$ are the initial and terminal densities.
\end{description}
  \addtolength{\itemindent}{0.25in} These choices are inherited from the code for the unconstrained case by \cite{chizat2018interpolating} and have been found to work well in practice.
For each experiment, we also have specific hyperparameters, which are summarized in Table~\ref{tab:experiment_hyperparameters}.

\begin{table}[!htbp]
\centering
\small            %
\setlength{\tabcolsep}{1pt}
\begin{tabular}{lcccc}
\hline
\textbf{Experiment} & $\mathbf{T}$ & \textbf{Space grid size} & \#iter & $\delta$ \\
\hline
\texttt{SHK}            & 15 & $(256,)$     & 10000 & 1.0 \\
\texttt{Inequality total mass} & 15 & $(256,)$     & 3000  & $\tfrac{1}{2\pi}$ \\
\texttt{2D total mass}        & 15 & $(30,\,30)$ & 3000  & $1.0$ \\
\texttt{Static barrier}        & 30 & $(30,\,30)$ & 7000  & $10.0$ \\
\texttt{Moving barrier}        & 30 & $(30,\,30)$ & 7000  & $10.0$ \\
\texttt{French population}                    & 12 & $(64,\,64)$ & 10000 & $1.0$ \\
\texttt{Symmetric curve}       & 15 & $(256,)$     & 10000 & $0.01$ \\
\texttt{Nonsymmetric curve}     & 15 & $(256,)$     & 10000 & $0.01$ \\
\texttt{River}                   & 15 & $(30,\,30)$ & 3000  & $2.0$ \\
\texttt{Budget}                 & 15 & $(256,)$     & 3000  & $\tfrac{1}{2\pi}$ \\
\texttt{TOIAM}                 & 16 & $(27,32)$     & 50000  & $0.05$ \\
\hline
\end{tabular}
\caption{Hyperparameters for each experiment. Here, $T$ is the number of time steps, \#iter is the number of iterations, and $\delta$ is the interpolation parameter between transport and Fisher-Rao terms in the WFR energy.}
\label{tab:experiment_hyperparameters}
\end{table}

\subsection{Wall-clock times}
The wall-clock times for each experiment are summarized in Table~\ref{tab:experiment_times}.
\begin{description}
  \addtolength{\itemindent}{0.25in}
  \item[Execution:] All experiments were run on a single CPU core.
  \item[Timing scope:] Optimization only, excluding data loading and visualization.
\end{description}
\begin{table}[!htbp]
\centering
\capstartfalse
\small            %
\setlength{\tabcolsep}{1pt}
\begin{tabular}{lc}
\hline
\textbf{Experiment} & Mean time (sec) $\pm$ std. dev. \\
\hline
\texttt{SHK}            & 22.025 $\pm$ 0.424 \\
\texttt{Inequality total mass} & 7.400 $\pm$ 0.320 \\
\texttt{2D total mass}        & 23.359 $\pm$ 0.621 \\
\texttt{Static barrier}        & 115.066 $\pm$ 3.094 \\
\texttt{Moving barrier}        & 114.071 $\pm$ 2.726 \\
\texttt{French population}                    & 425.274 $\pm$ 24.554 \\
\texttt{Symmetric curve}       & 24.916 $\pm$ 0.184 \\
\texttt{Nonsymmetric curve}     & 24.577 $\pm$ 0.067 \\
\texttt{River}                   & 26.738 $\pm$ 0.426 \\
\texttt{Budget}                 & 6.762 $\pm$ 0.016 \\
\texttt{toiam\_wfr}              & 285.279 $\pm$ 1.136 \\
\texttt{toiam\_no\_lineage}      & 521.731 $\pm$ 2.996 \\
\texttt{toiam\_lineage\_30pct}   & 572.628 $\pm$ 1.535 \\
\texttt{toiam\_lineage\_50pct}   & 632.736 $\pm$ 46.271 \\
\texttt{toiam\_lineage\_80pct}   & 550.241 $\pm$ 32.463 \\
\texttt{toiam\_full\_lineage}    & 490.656 $\pm$ 17.419 \\
\hline
\end{tabular}
\caption{Wall-clock times for each experiment. The mean is computed over 3 runs.}
\label{tab:experiment_times}
\end{table}

\section{Pseudocode of the PPXA Implementation of Constrained UOT}\label{appendix:algo}
\begin{algorithm}[H]\caption{PPXA for constrained WFR geodesics with $M$ box constraints}
\label{alg:ppxa}
\begin{algorithmic}[1]
\State \textbf{Input:} \(\rho_0,\rho_1\), grid sizes \(N_0,\ldots,N_d\), lengths \(L_1,\ldots,L_d\).
\State \textbf{Constraints:} For each \(m=1,\dots,M\), weights \(H^{(m)}_\bullet\) and bounds \((\ell^{(m)},u^{(m)})\).
\State Initialize \((U,V)\) (linear or Fisher–Rao path for \(\rho\)); set \(y_1=\cdots=y_{2+M}=(U,V)\).
\State Set \(\pi_k=(U,V)\) for \(k=1,\dots,2+M\).
\For{\(n=0,1,\dots,n_{\max}-1\)}
  \State \(\pi_1\gets\textrm{prox}_{\gamma(J+\iota_{\mathcal{CE}})}(y_1)\) \Comment{pointwise WFR prox on \(V\) + CE projection on \(U\)}
  \State \(\pi_2\gets\textrm{prox}_{\iota_{\{V=I(U)\}}}(y_2)\) \Comment{projection onto \(V=I(U)\)}
  \For{\(m=1\) \textbf{to} \(M\)}
    \State \(\pi_{m+2}\gets\textrm{prox}_{\iota_{\mathcal{C}^{(m)}}}(y_{m+2})\) \Comment{affine box projection (constraint \(m\))}
  \EndFor
  \State \(\bar\pi\gets\big(\pi_1+\pi_2+\sum_{m=1}^M \pi_{m+2}\big)/(2+M)\)
  \State \(y_k\gets y_k+\alpha(2\bar\pi - (U,V)-\pi_k)\) for \(k=1,\dots,2+M\)
  \State \((U,V)\gets (U,V)+\alpha(\bar\pi-(U,V))\)
\EndFor
\State \textbf{Return} \((U,V)\) with either $\bar{\rho}$ or $\rho$ as the geodesic where $U=(\bar\rho,\bar\omega,\bar\zeta)$ or \(V=(\rho,\omega,\zeta)\).
\end{algorithmic}
\end{algorithm}

\section{Convergence analysis of the PPXA algorithm}
\label{sec:convergence_analysis}
We provide here some empirical evaluation regarding the convergence rate of the PPXA algorithm. In  Figure \ref{fig:convergence_analysis}, we specifically compute the relative error between the current iterate and the last iterate of the PPXA algorithm for the SHK example (Section \ref{sec:SHK}) and the static barrier example (Section \ref{subsec:barrier_constraint}) and plot those in the log scale. We have fitted linear curves to the plots that provide good agreement with the empirical convergence rate: for the SHK example we obtained a convergence rate of $q=0.999414$ for the unconstrained case and $q=0.999472$ for the constrained case. The corresponding coefficients of determination (the fraction of variance in the data explained by the model) for the linear model were $R^2 =0.978$ for the unconstrained case and $R^2=0.974$ for the constrained case.  For the static barrier example, 
we obtained a convergence rate of $q=0.999568$ for the unconstrained case and $q=0.999525$ for the constrained case. Here the corresponding coefficients of determination were $R^2=0.966$ for the unconstrained case and $R^2=0.979$ for the constrained case. 

\begin{figure}[htbp]
  \begin{center}
    \begin{subfigure}[b]{0.48\textwidth}
      \includegraphics[width=\linewidth]{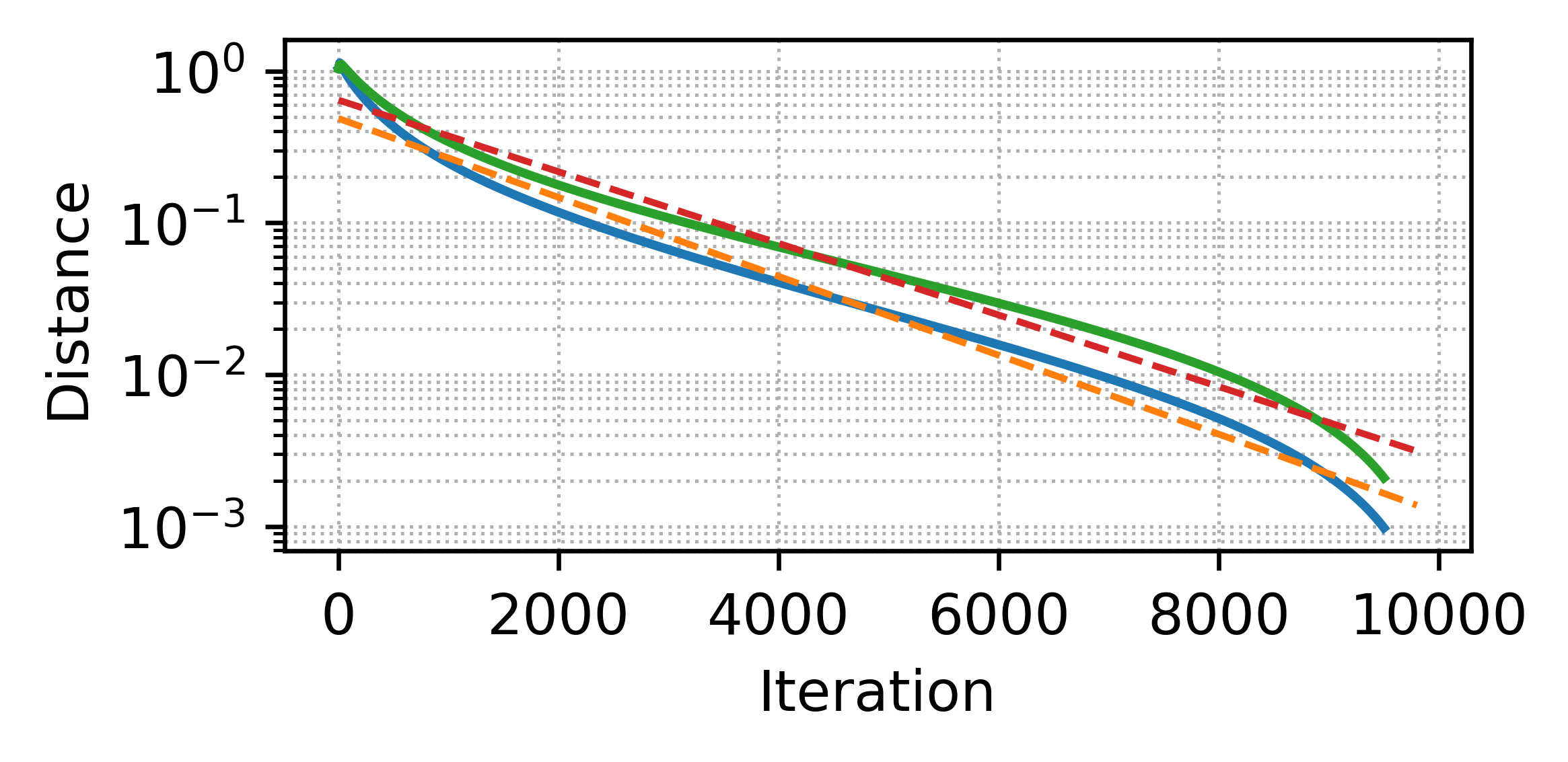}
      \caption{}
      \label{fig:convergence_analysis_left}
    \end{subfigure}
    \hfill
    \begin{subfigure}[b]{0.48\textwidth}
      \includegraphics[width=\linewidth]{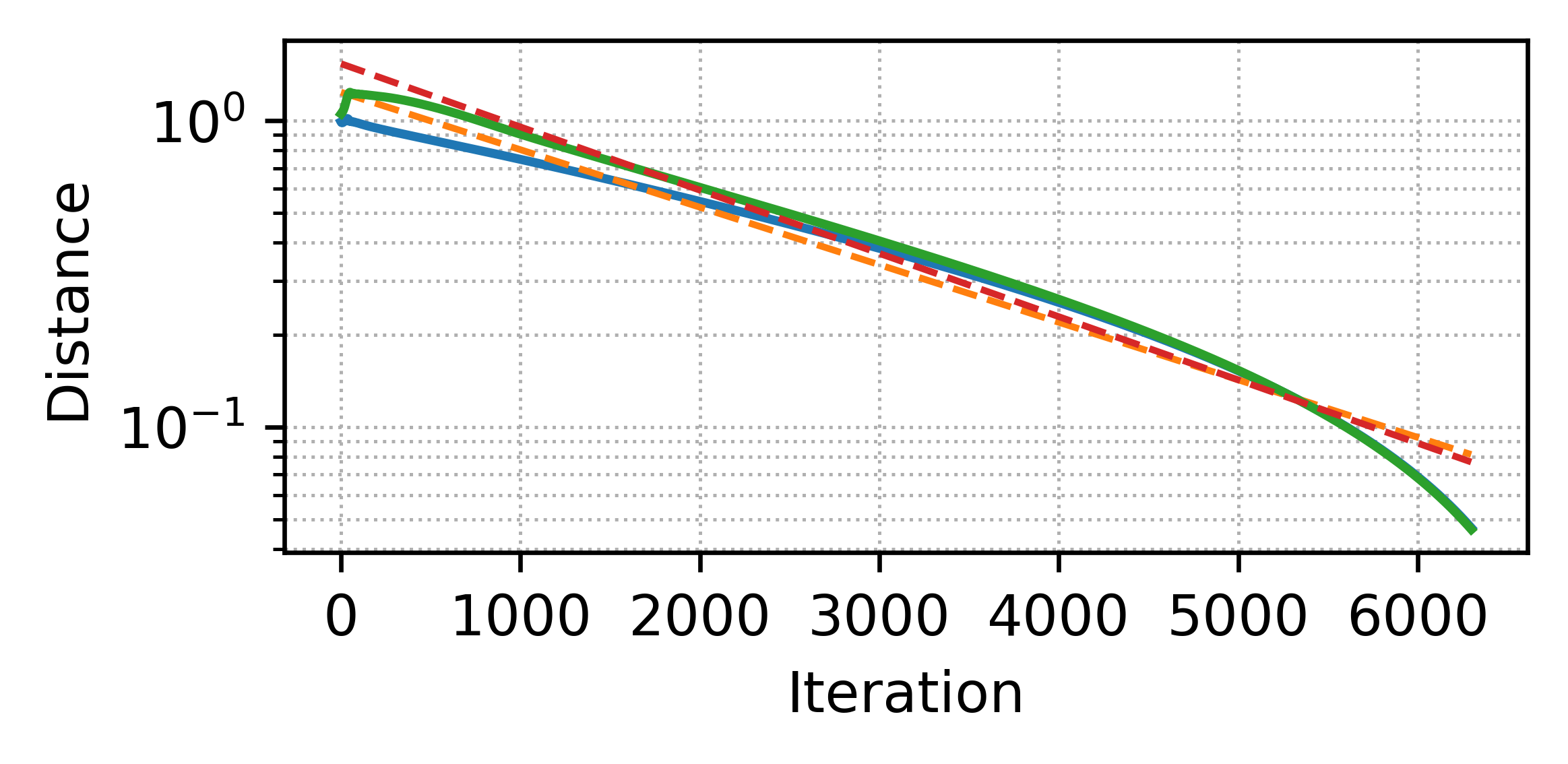}
      \caption{}
      \label{fig:convergence_analysis_right}
    \end{subfigure}
  \end{center}
  \caption{Convergence plot of the PPXA algorithm. Panel \subref{fig:convergence_analysis_left} shows the SHK experiment (Section \ref{sec:SHK}) while panel \subref{fig:convergence_analysis_right} shows the static barrier example (Section \ref{subsec:barrier_constraint}). The $y$-axis shows the relative error between the current iterate and the last iterate (see Algorithm \ref{alg:ppxa}) in the $L^2$ norm on a log scale. In both figures, the {dashed} curves are the obtained linear fits.} %
  \label{fig:convergence_analysis}
\end{figure}

We note that in all these cases the estimated $q$ is very close to one, indicating very slow convergence, e.g. for the SHK example, we need $-1/\log_{10}q \sim$ 3800-4000 iterations to reduce the error by 1 decimal place.   Overall, this convergence analysis indicates that while the PPXA algorithm for unbalanced optimal transport problems converges formally geometrically, the rate constant is almost one and thus the practical convergence is rather sublinear. Furthermore, this convergence rate is not significantly affected by the presence of constraints.

\section{Wasserstein-Fisher-Rao Minimizers Between Symmetric Measures}\label{sec:proof_lem_symmetric}
\begin{lemma}
  \label{thm:symmetric_solution}
  Let $\rho_0, \rho_1$ be symmetric with respect to the map $R(x) = -x$. Then, there exists a minimizer of the unconstrained WFR problem between $\rho_0$ and $\rho_1$ that is symmetric at almost all times $t\in [0,1]$. In particular, this minimizer satisfies the constraint \eqref{eq:convex_set_constraint}.
\end{lemma}

\begin{proof}
  From the results of \cite{chizat2018interpolating}, there exists a minimizer $(\rho,\omega,\zeta)$ to the standard WFR problem where $\rho,\omega,\zeta \in \mathcal{M}([0,1]\times \mathbb{S}^1)$. With Propositions 2.2 and 2.6 of \cite{Bredies_2020}, we can further assume that all three measures are disintegrable in time with respect to the Lebesgue measure. Let us then define the following: 
  \begin{align}
  &\bar{\rho} =  \frac{1}{2}(\rho_t + R_{*}\rho_t) \otimes dt\\
  &\bar{\omega}= \frac{1}{2}(\omega_t + R_{*}\omega_t) \otimes dt \\
  &\bar{\zeta} = \frac{1}{2}(\zeta_t + R_{*}\zeta_t) \otimes dt. 
  \end{align}
  where $dt$ represents the Lebesgue measure on $[0,1]$ and $\otimes$ represents the disintegration of measure. Due to the symmetry of $\rho_0$ and $\rho_1$, we see that $\bar{\rho}_0 = \rho_0$ and $\bar{\rho}_1=\rho_1$. In addition, for every $\phi \in \mathcal{C}^1([0,1]\times \mathbb{S}^1)$,
  \begin{align}
    &\int_{[0,1]\times \mathbb{S}^1} \partial_{t}\phi d\bar{\rho} + \int_{[0,1]\times \mathbb{S}^1}\partial_{\theta}\phi d\bar{\omega} +\int_{[0,1]\times \mathbb{S}^1}\phi d\bar{\zeta} \\ 
    &= \frac{1}{2} \left(\int_{[0,1]\times \mathbb{S}^1} \partial_{t}\phi d \rho + \int_{[0,1]\times \mathbb{S}^1} \partial_{\theta}\phi d\omega + \int_{[0,1]\times \mathbb{S}^1} \phi d \zeta \right) \\
    &+ \frac{1}{2} \left( \int_{0}^1 \int_{\mathbb{S}^1} \partial_{t}\phi d (R_{*}\rho_t) dt + \int_{0}^1 \int_{\mathbb{S}^1} \partial_{\theta}\phi d(R_{*}\omega_t) dt + \int_{0}^1 \int_{\mathbb{S}^1}\phi d (R_{*}\zeta_t) dt \right).
  \end{align}
  By assumption the upper term equals to $\frac{1}{2}\left(\int_{\mathbb{S}^1}\phi(1,\cdot)d\rho_1 - \int_{\mathbb{S}^1}\phi(0,\cdot)d\rho_0 \right)$. By a change of variable in the inner integrals, writing $\bar{\phi}(t,\theta) = \phi(t,\theta + \pi)$, we obtain that the second term becomes:
  \begin{align}
     &\frac{1}{2} \left(\int_{0}^1 \int_{\mathbb{S}^1} \partial_{t}\bar{\phi} d \rho_t dt + \int_{0}^1 \int_{\mathbb{S}^1} \partial_{\theta}\bar{\phi} d\omega_t dt + \int_{0}^1 \int_{\mathbb{S}^1} \bar{\phi} d\zeta_t dt  \right) \\
     &=\frac{1}{2} \left(\int_{\mathbb{S}^1}\bar{\phi}(1,\cdot)d\rho_1 - \int_{\mathbb{S}^1}\bar{\phi}(0,\cdot)d\rho_0   \right) =\frac{1}{2} \left(\int_{\mathbb{S}^1}\phi(1,\cdot)d\rho_1 - \int_{\mathbb{S}^1}\phi(0,\cdot)d\rho_0   \right) 
  \end{align}
  where the last equality follows from the fact that $\rho_0$ and $\rho_1$ are symmetric. It results that $(\bar{\rho},\bar{\omega},\bar{\zeta})$ satisfy the continuity equation. Furthermore, for a.e. $t\in[0,1]$, letting $\lambda_t$ be a measure such that $\bar{\rho}_t,\bar{\omega}_t,\bar{\zeta}_t \ll \lambda_t$, we have:
  \begin{align}
    &\int_0^1 \int_{\mathbb{S}^1} f_{\delta}\left(\frac{d\bar{\rho}_t}{d\lambda_t},\frac{d\bar{\omega}_t}{d\lambda_t},\frac{d\bar{\zeta}_t}{d\lambda_t} \right) d\lambda_t dt \\
    &\leq \frac{1}{2} \int_0^1 \int_{\mathbb{S}^1} f_{\delta}\left(\frac{d\rho_t}{d\lambda_t},\frac{d\omega_t}{d\lambda_t},\frac{d\zeta_t}{d\lambda_t} \right) d\lambda_t dt + \frac{1}{2} \int_0^1 \int_{\mathbb{S}^1} f_{\delta}\left(\frac{dR_{*}\rho_t}{d\lambda_t},\frac{dR_{*}\omega_t}{d\lambda_t},\frac{dR_{*}\zeta_t}{d\lambda_t} \right) d\lambda_t dt
  \end{align}
  which follows from the convexity of the function $f_{\delta}$. Note finally that, due to the invariances of the cost function, $(R_{*}\rho_t,R_{*}\omega_t,R_{*}\zeta_t) \otimes dt$ must be a minimizer for the WFR$_\delta$ distance between $R_{*}\rho_0 = \rho_0$ and $R_{*}\rho_1 = \rho_1$, and as a consequence both terms are equal to $\text{WFR}_{\delta}(\rho_0,\rho_1)^2/2$. We thus get that:
  \begin{equation}
    \int_0^1 \int_{\mathbb{S}^1} f_{\delta}\left(\frac{d\bar{\rho}_t}{d\lambda_t},\frac{d\bar{\omega}_t}{d\lambda_t},\frac{d\bar{\zeta}_t}{d\lambda_t} \right) d\lambda_t dt \leq \text{WFR}_{\delta}(\rho_0,\rho_1)^2 
  \end{equation}
  and thus $t \mapsto\bar{\rho}_t$ is a geodesic between $\rho_0$ and $\rho_1$ and is symmetric at all times.
\end{proof}

\clearpage
\section{TOIAM momentum and source fields}
\label{appendix:toiam_controls}
Figures~\ref{fig:toiam_controls_200_250}--\ref{fig:toiam_controls_550_600} show the inferred momentum $\omega$ and source term $\zeta$ for TOIAM sequence 00 over frames 200--600, with $\delta=0.05$. Each figure covers two consecutive 25-frame intervals and compares unconstrained WFR with ACUOT using moment constraints and 0, 30, 50, 80, or 100\% of the lineage information. Cyan outlines mark new cells; red and blue indicate positive and negative source values, respectively. Arrows represent momentum, with color indicating its magnitude. The panels use normalized spatial units, symmetric logarithmic color scales, and logarithmic arrow lengths.

\begin{figure}[!htbp]
    \centering
    \includegraphics[width=\textwidth,height=0.72\textheight,keepaspectratio]{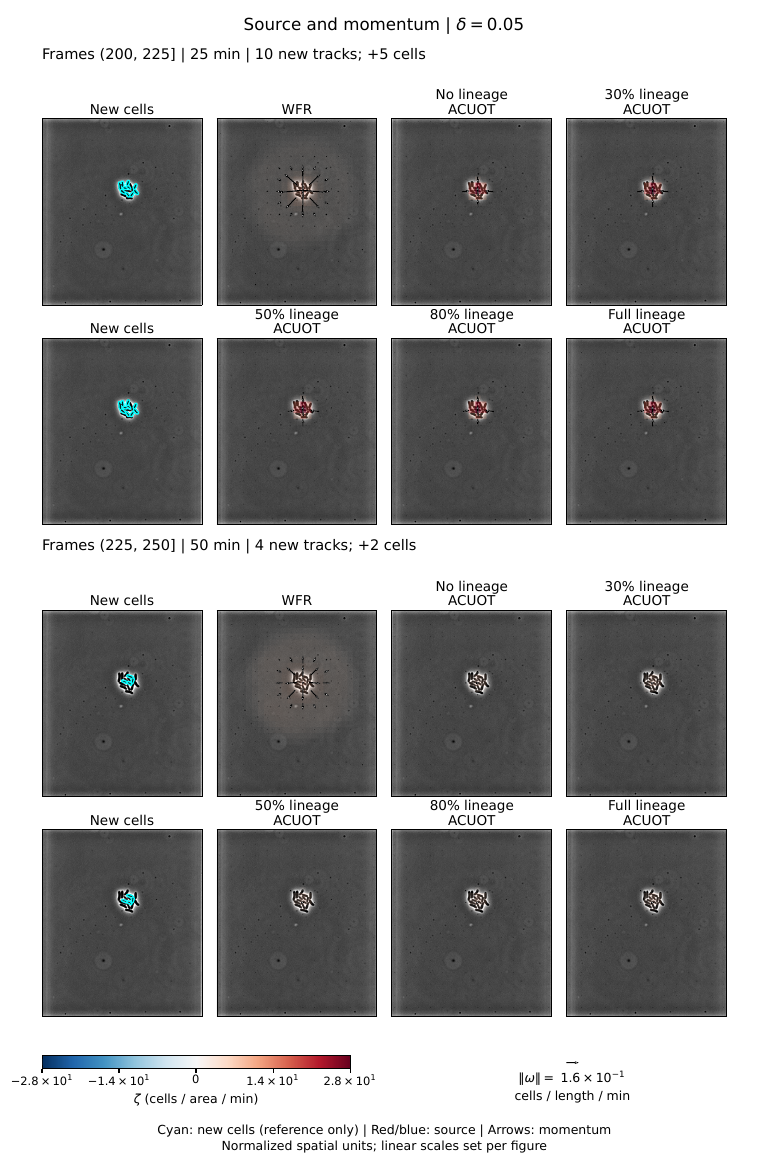}
    \caption{Inferred source and momentum for TOIAM sequence 00, frames 200--250, with $\delta=0.05$. The upper and lower groups show frames $(200,225]$ and $(225,250]$, respectively.}
    \label{fig:toiam_controls_200_250}
\end{figure}

\clearpage

\begin{figure}[!htbp]
    \centering
    \includegraphics[width=\textwidth,height=0.88\textheight,keepaspectratio]{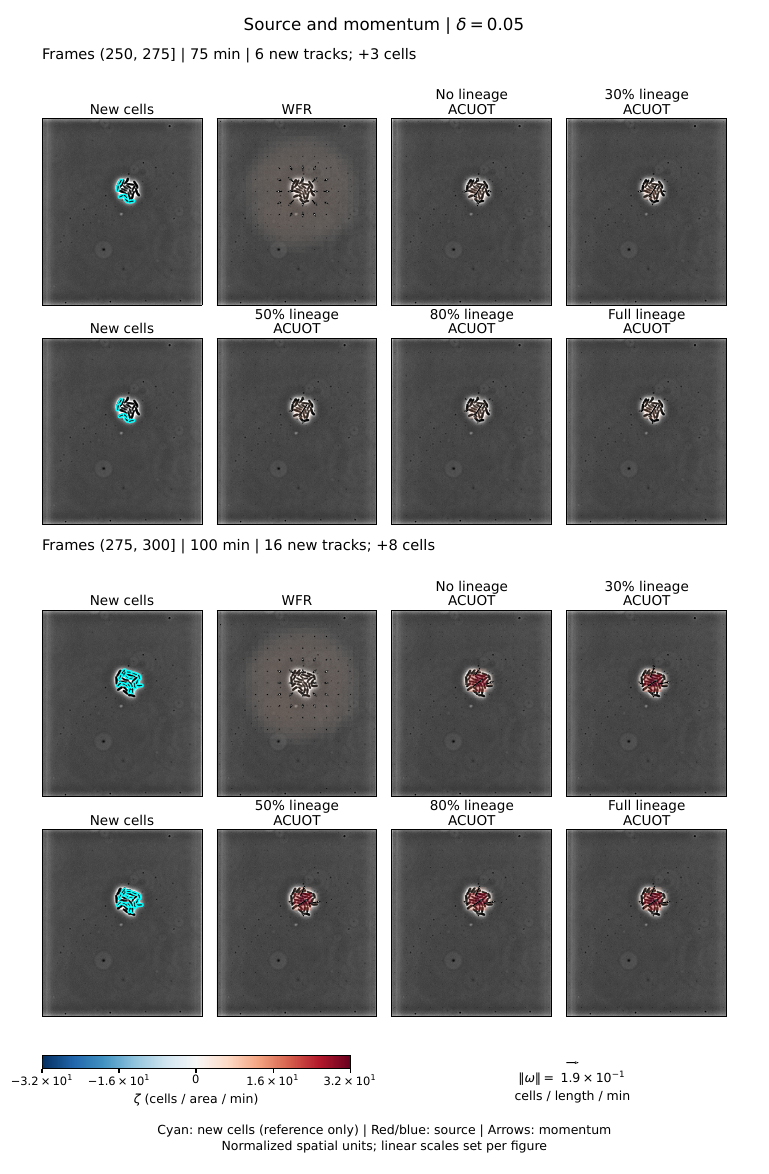}
    \caption{Inferred source and momentum for TOIAM sequence 00, frames 250--300, with $\delta=0.05$. The upper and lower groups show frames $(250,275]$ and $(275,300]$, respectively.}
    \label{fig:toiam_controls_250_300}
\end{figure}

\clearpage

\begin{figure}[!htbp]
    \centering
    \includegraphics[width=\textwidth,height=0.88\textheight,keepaspectratio]{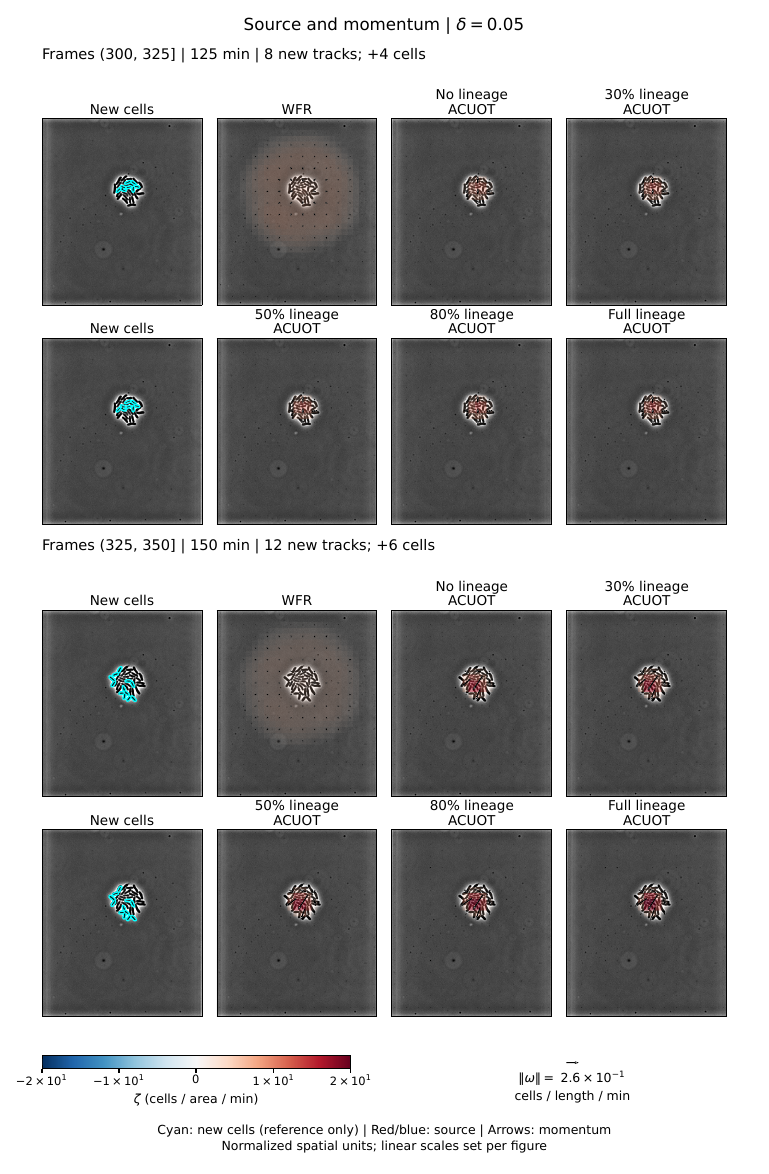}
    \caption{Inferred source and momentum for TOIAM sequence 00, frames 300--350, with $\delta=0.05$. The upper and lower groups show frames $(300,325]$ and $(325,350]$, respectively.}
    \label{fig:toiam_controls_300_350}
\end{figure}

\clearpage

\begin{figure}[!htbp]
    \centering
    \includegraphics[width=\textwidth,height=0.88\textheight,keepaspectratio]{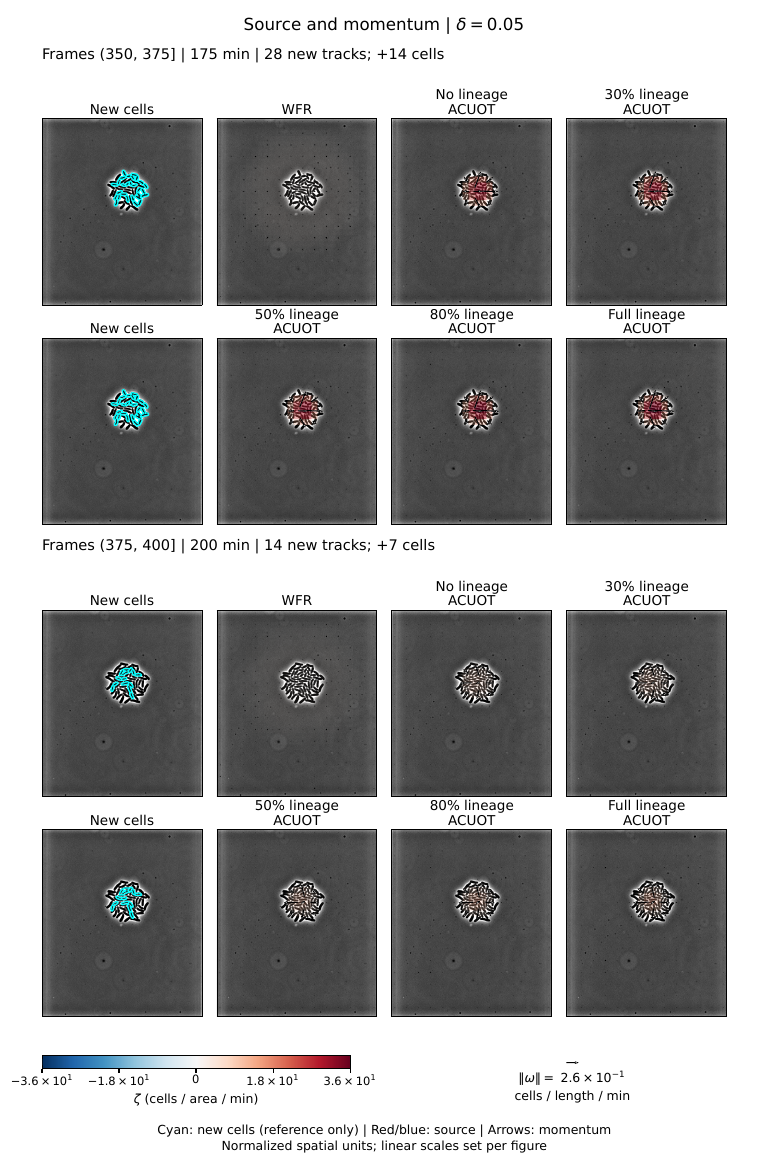}
    \caption{Inferred source and momentum for TOIAM sequence 00, frames 350--400, with $\delta=0.05$. The upper and lower groups show frames $(350,375]$ and $(375,400]$, respectively.}
    \label{fig:toiam_controls_350_400}
\end{figure}

\clearpage

\begin{figure}[!htbp]
    \centering
    \includegraphics[width=\textwidth,height=0.88\textheight,keepaspectratio]{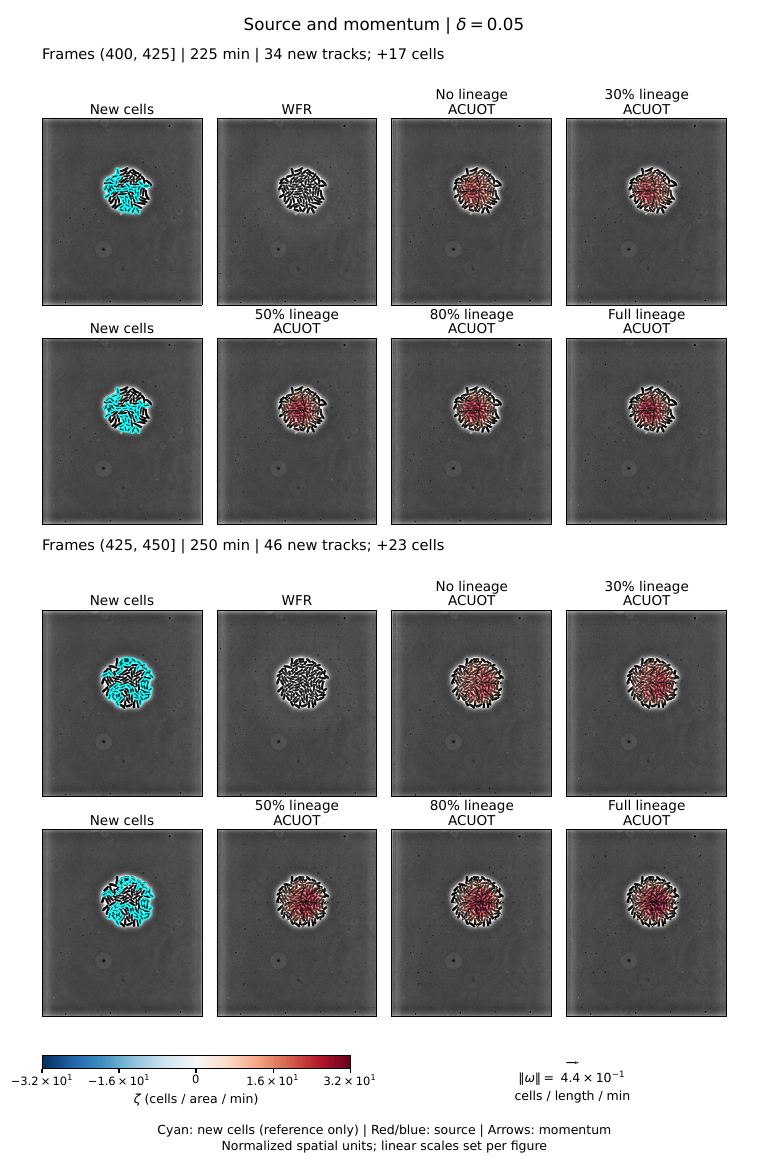}
    \caption{Inferred source and momentum for TOIAM sequence 00, frames 400--450, with $\delta=0.05$. The upper and lower groups show frames $(400,425]$ and $(425,450]$, respectively.}
    \label{fig:toiam_controls_400_450}
\end{figure}

\clearpage

\begin{figure}[!htbp]
    \centering
    \includegraphics[width=\textwidth,height=0.88\textheight,keepaspectratio]{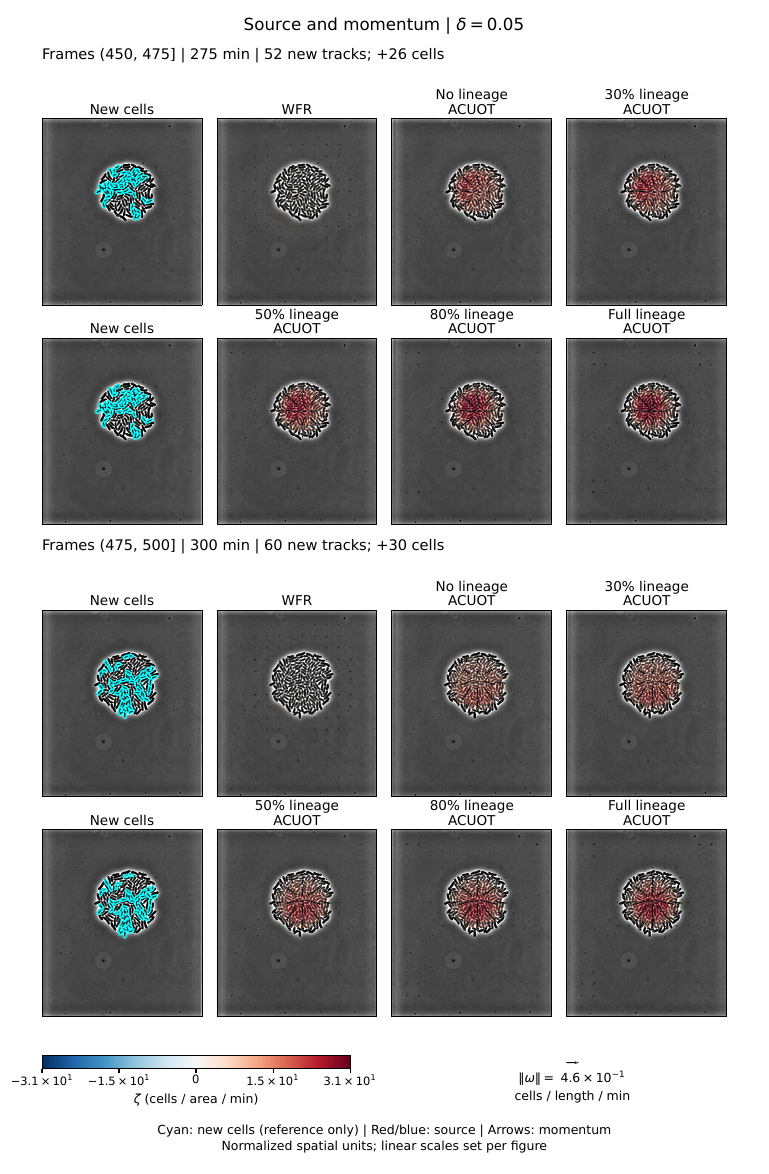}
    \caption{Inferred source and momentum for TOIAM sequence 00, frames 450--500, with $\delta=0.05$. The upper and lower groups show frames $(450,475]$ and $(475,500]$, respectively.}
    \label{fig:toiam_controls_450_500}
\end{figure}

\clearpage

\begin{figure}[!htbp]
    \centering
    \includegraphics[width=\textwidth,height=0.88\textheight,keepaspectratio]{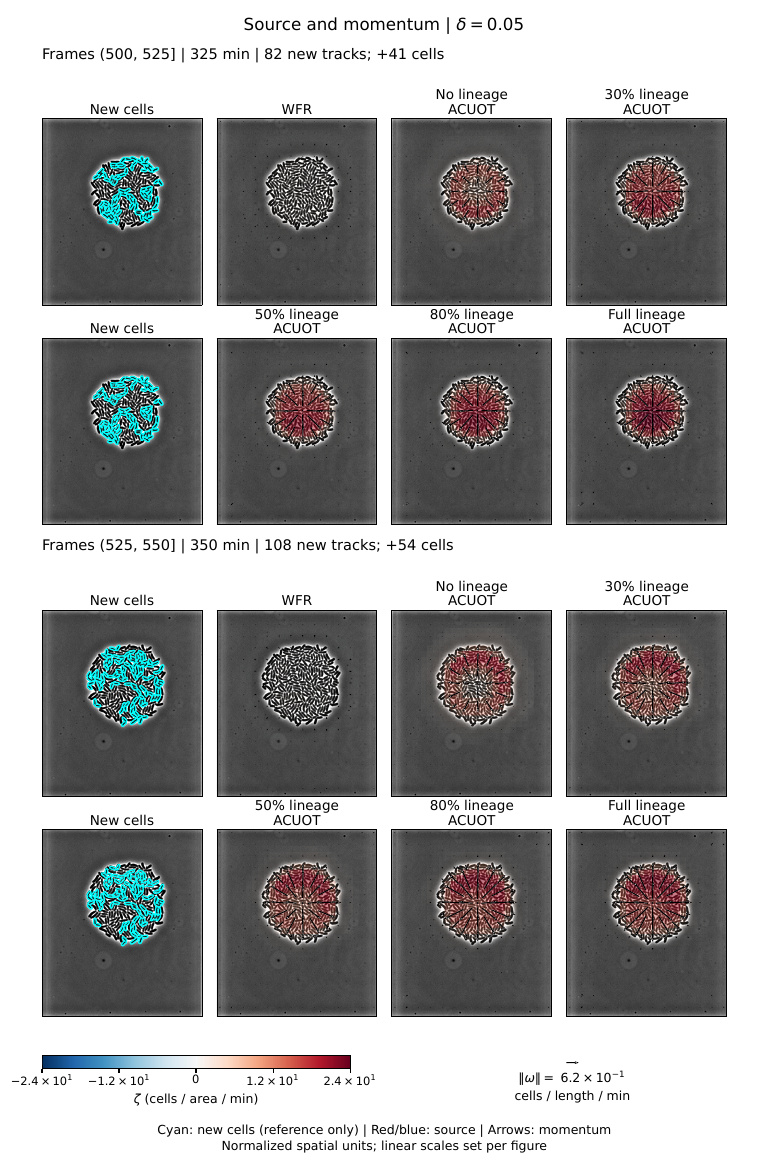}
    \caption{Inferred source and momentum for TOIAM sequence 00, frames 500--550, with $\delta=0.05$. The upper and lower groups show frames $(500,525]$ and $(525,550]$, respectively.}
    \label{fig:toiam_controls_500_550}
\end{figure}

\clearpage

\begin{figure}[!htbp]
    \centering
    \includegraphics[width=\textwidth,height=0.88\textheight,keepaspectratio]{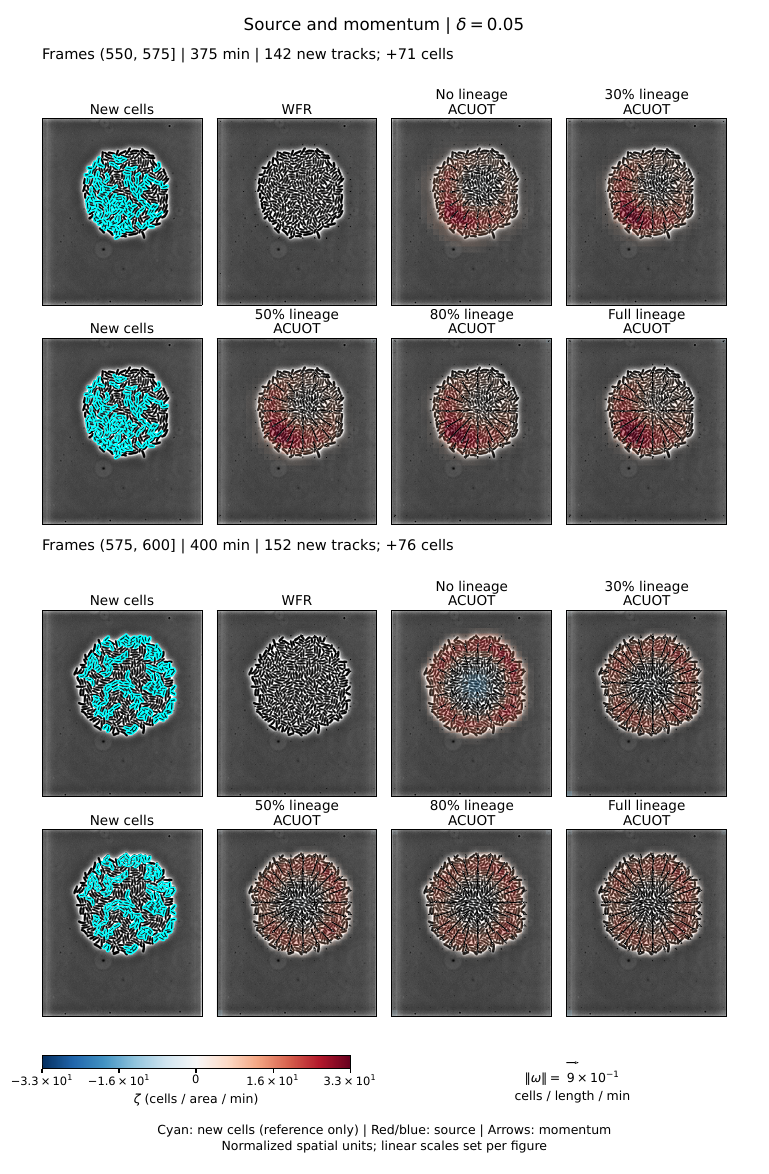}
    \caption{Inferred source and momentum for TOIAM sequence 00, frames 550--600, with $\delta=0.05$. The upper and lower groups show frames $(550,575]$ and $(575,600]$, respectively.}
    \label{fig:toiam_controls_550_600}
\end{figure}

\clearpage

\end{document}